\documentclass[10pt]{article}
\usepackage{xcolor}

\usepackage{physics}
\usepackage{amsmath}
\usepackage{amssymb}
\usepackage{amsthm}
\usepackage{fullpage}
\usepackage{graphics}
\usepackage{graphicx}
\usepackage{subcaption}
\usepackage{caption}
\usepackage[font=footnotesize,labelfont=bf]{caption}
\usepackage[draft]{todonotes} 
\usepackage{mathtools}
\usepackage{bbm}
\usepackage{multicol}
\usepackage[shortlabels]{enumitem}
\usepackage[normalem]{ulem}

\usepackage{tikz}
\usetikzlibrary{shapes.multipart}
\usetikzlibrary{patterns}
\usetikzlibrary{shapes.multipart}
\usetikzlibrary{arrows}
\usetikzlibrary{decorations.markings}

\newlist{abbrv}{itemize}{1}
\setlist[abbrv,1]{label=,labelwidth=1in,align=parleft,itemsep=0.1\baselineskip,leftmargin=!}

\newtheorem{thm}{Theorem}[section]

\newtheorem{lm}{Lemma}[section]
\newtheorem{prop}{Proposition}[section]
\newtheorem{rmk}{Remark}[section]

\newcommand{\prob}{\mathbb{P}}
\newcommand{\intZ}{\mathbb{Z}}
\newcommand{\realR}{\mathbb{R}}

\newcommand{\FGUE}{F_{GUE}}

\newcommand{\iid}{i.i.d.\,}

\newcommand{\dist}{{\rm{dist}}\,}

\newcommand{\mr}{\mathbf}

\renewcommand{\dd}{{\mathrm d}}

\newcommand{\ttt}{\mathfrak{t}}

\newcommand{\beq}{ \begin{equation} }
\newcommand{\eeq}{ \end{equation} }

\newcommand{\param}{\kappa_1}
\newcommand{\params}{\kappa_2}
\newcommand{\X}{y}
\newcommand{\ara}{Q}
\newcommand{\arao}{R}
\newcommand{\aras}{C_1}
\newcommand{\arat}{C_2}
\renewcommand{\ttt}{T}

\numberwithin{equation}{section} 
\numberwithin{thm}{section}

\author{Jinho Baik\footnote{Department of Mathematics, University of Michigan,
Ann Arbor, MI, 48109. Email: \texttt{baik@umich.edu}} 
 and Zhipeng Liu\footnote{Courant Institute of Mathematical Sciences, New York University, New York, NY 10012. Email: \texttt{zhipeng@cims.nyu.edu}}}
\date{\today}

\begin{document}
\title{TASEP on a ring in sub-relaxation time scale}
\maketitle

\begin{abstract}
Interacting particle systems in the KPZ universality class on a ring of size $L$ with $O(L)$ number of particles are expected to change from KPZ dynamics to equilibrium dynamics at the
so-called relaxation 
time scale $t=O(L^{3/2})$. 
In particular the system size is expected to have little effect to the particle fluctuations in the sub-relaxation time scale $t\ll L^{3/2}$. 
We prove that this is indeed the case for the totally asymmetric simple exclusion process (TASEP) with two types of initial conditions. 
For flat initial condition, we show that the particle fluctuations are given by the Airy$_1$ process as in the infinite TASEP with flat initial condition. 
On the other hand, the TASEP on a ring with step initial condition is equivalent to the periodic TASEP with a certain shock initial condition. 
We compute the fluctuations explicitly both away from and near the shocks for the infinite TASEP with same initial condition, and then show that the periodic TASEP has same fluctuations in the sub-relaxation time scale. 
\end{abstract}

\section{Introduction and results}


Consider the fluctuations of the particle locations of a one-dimensional interacting particle system in the KPZ class (such as asymmetric simple exclusion process) on a ring of size $L$ with $N=O(L)$ number of particles.   
The system behaves for a while like an infinite system, but eventually the dynamics will be influenced by the system size $L$. 
Since the spatial correlations of the infinite system are expected to be of order $t^{2/3}$, 
the fluctuations of all particles become correlated if $t^{2/3}=O(L)$. 
This time scale, $t=O(L^{3/2})$, is referred as the relaxation time scale   
\cite{Gwa-Spohn92, Derrida-Lebowitz98, LeeKim06, Brankov-Papoyan-Poghosyan-Priezzhev06, Gupta-Majumdar-Godreche-Barma07, Proeme-Blythe-Evans11}. 
Recently the one-point limit law in this relaxation time scale was obtained in \cite{Prolhac16, Baik-Liu16} for the totally asymmetric simple exclusion process (TASEP) on a ring. 
It was shown that the height fluctuations are still of order $t^{1/3}$ as in the KPZ scaling but the limiting distributions are something new, different from the Tracy-Widom distributions.
The goal of this paper is to study the effect of the system size in the sub-relaxation time scale $t\ll L^{3/2}$. 
We focus on the TASEP on a ring with flat and step initial conditions and the results will complement the papers \cite{Prolhac16, Baik-Liu16}.

The TASEP on a ring is equivalent to the \emph{periodic TASEP}, and we present our results in this model. 
In the periodic TASEP of period $L$ with $N$ particles, the particles are on the integers $\intZ$ and satisfy the periodicity $x_k(t)=x_{k+N}(t)+L$ for all $k\in\intZ$ and $t\ge 0$.
In the usual TASEP, the particles have independent clocks of exponential waiting time with parameter $1$. A particle moves to its right neighboring site if its clock rings and the right neighboring site is empty. 
For the periodic TASEP, the clocks corresponding to the particles at the sites of $L$ distance apart are identical. Hence if a particle at site $i$ moves, then the particles at sites $i+nL$, $n\in \intZ$, all move. 
Any $N$ consecutively-indexed particles in the periodic TASEP then describe the TASEP on a ring with the additional information of 
the winding numbers around the ring. 
An initial condition for the TASEP on a ring introduces a periodic initial condition for the periodic TASEP. 
The flat initial condition on a ring translates to the flat initial condition for the periodic TASEP: for a fixed positive integer $d\ge 2$, 
\begin{equation}
\label{eq:flat_ic000}
	x_j(0)=jd, \qquad j\in \intZ. 
\end{equation}
(We label the particles left to right, and they move from left to right.)
However, the step initial condition on a ring translates to the following initial condition for the periodic TASEP: 
\begin{equation}
\label{eq:step_ic00}
	x_{j+kN}(0)=-N+j+kL, \qquad j=1,2,\cdots, N, \qquad k\in \intZ.
\end{equation}
See the first picture in Figure~\ref{fig:densityprofile}. 
Note that this is a shock initial condition. 
For both initial conditions, 
we keep the average particle density $\rho$ as a constant and take $t\to\infty$ and $L=[\rho^{-1}N]\to\infty$ simultaneously.
We compare the periodic TASEP with  the usual TASEP on $\intZ$, which we call \emph{infinite TASEP}, with same initial condition. 
The main result is that as long as $t\ll L^{3/2}$,  the particle fluctuations of the periodic TASEP are same as those of the infinite TASEP.
Hence the system size has little effect in the  sub-relaxation time scale. 

\subsection{Flat initial condition}

For the infinite TASEP with flat initial condition, the particle fluctuations converge to the Airy$_1$ process $\mathcal{A}_1(u)$  \cite{Sasamoto05, Borodin-Ferrari-Prahofer-Sasamoto07}.
The marginals of $\mathcal{A}_1(u)$ are distributed as the GOE Tracy-Widom distribution \cite{Tracy-Widom96}. 
We show that the periodic TASEP with flat initial condition has the same fluctuations if $t\ll L^{3/2}$.

Let the average density $\rho$ of particles be a fixed constant satisfying 
\begin{equation}\label{eq:flatrhocoq}
	\rho\in\{d^{-1} \, ; \, d=2,3,4,\cdots\}.
\end{equation}
Define the flat initial condition as
\begin{equation}
\label{eq:flat_ic0}
	x_j(0)=j\rho^{-1}, \qquad j=1,2,\cdots, N
\end{equation}
for the TASEP on a ring (where the ring is identified with the set $\{1, 2, \cdots, L\}$ with $L=N\rho^{-1}$).
Then the corresponding periodic TASEP satisfies the flat initial condition, 
\begin{equation}
\label{eq:flat_ic}
	x_j(0)=j\rho^{-1}, \qquad j\in \intZ.
\end{equation}

\begin{thm}
	\label{thm:limiting_process_flat} (flat initial condition)
	Fix $\rho$ satisfying~\eqref{eq:flatrhocoq} and consider the periodic TASEP of period $L=\rho^{-1}N$ with $N$ particles.
	Assume the flat initial condition~\eqref{eq:flat_ic}.  Let $t=t_N$ be a sequence of times satisfying the following two conditions: (1) $t_N\le CN^{3/2-\epsilon}$ for fixed positive constants $C$ and $\epsilon$, and (2) $\lim_{N\to\infty} t_N=\infty$. 
	Then setting
\begin{equation}
\label{eq:def_kappa_sigma_flat}
\param:=2^{5/3}\rho^{4/3}(1-\rho)^{1/3}, \qquad \sigma_1:=2^{1/3}\rho^{-1/3}(1-\rho)^{2/3},
\end{equation}
	we have
	\begin{equation}
	\label{eq:limiting_process_flat}
	\frac{x_{[\param ut_N^{2/3}]}(t_N)-x_{[\param ut_N^{2/3}]}(0)-(1-\rho)t_N}{-\sigma_1t_N^{1/3}} \longrightarrow  \mathcal{A}_1(u)
	\end{equation}
	for $u\in\realR$ in the sense of convergence of finite dimensional distributions as $N\to\infty$.
\end{thm}

Compare the above result with the infinite TASEP with the same initial condition~\eqref{eq:flat_ic}.
In this case, the result~\eqref{eq:limiting_process_flat} holds as $t\to\infty$ in an arbitrary way \cite{Sasamoto05, Borodin-Ferrari-Prahofer-Sasamoto07}.\footnote{\label{ft:flat_TASEP}The paper \cite{Borodin-Ferrari-Prahofer-Sasamoto07} only states the result for $\rho=1/2$, but the more general case $\rho=1/d$, $d=2,3,4,\cdots$, is similar. See \cite{Borodin-Ferrari-Prahofer07} for a discrete version. We note that the more general case $\rho=p/q$ for integer $p$ and $q$, however,  is still open for both infinite TASEP and TASEP on a ring.} 
The sub-relaxation condition (1) $t_N\le CN^{3/2-\epsilon}$ is not needed. 
(The infinite TASEP does not even depend $N$ and $L$ individually; it  depends only on the ratio $\rho=N/L$.)
On the other hand,  the one-point distribution of the periodic TASEP is not the GOE Tracy-Widom distribution when $t=O(L^{3/2})$ \cite{Prolhac16, Baik-Liu16}, and when $t\gg L^{3/2}$, we expect Gaussian fluctuations. 
Hence for the flat initial condition, the periodic TASEP has same fluctuations as the infinite TASEP only in the sub-relaxation time scale $t\ll L^{3/2}$.

\subsection{Periodic step initial condition} 



Define the step initial condition for the TASEP on a ring (where the ring is identified with $\{-N+1, -N+2, \cdots, -N+L\}$) as 
\begin{equation}
\label{eq:step_ic0}
	x_j(0)=-N+j, \qquad j=1,2,\cdots, N. 
\end{equation}
Then the corresponding periodic TASEP satisfies 
\begin{equation}
\label{eq:step_ic}
	x_{j+kN}(0)=-N+j+kL, \qquad j=1,2,\cdots, N, \qquad k\in \intZ.
\end{equation}
We call this \emph{periodic step initial condition.}
See the first picture in Figure~\ref{fig:densityprofile}.
Fix $\rho\in (0,1)$. We assume that $L=[\rho^{-1} N]$. Hence $\rho$ is the average density of particles. 
For simplicity we assume that $0<\rho\le 1/2$. 
See Remark  \ref{rmk:rhob2} for the case when $1/2< \rho<1$.

\begin{figure}
	\centering
	\begin{minipage}{.4\textwidth}
		\includegraphics[scale=0.35]{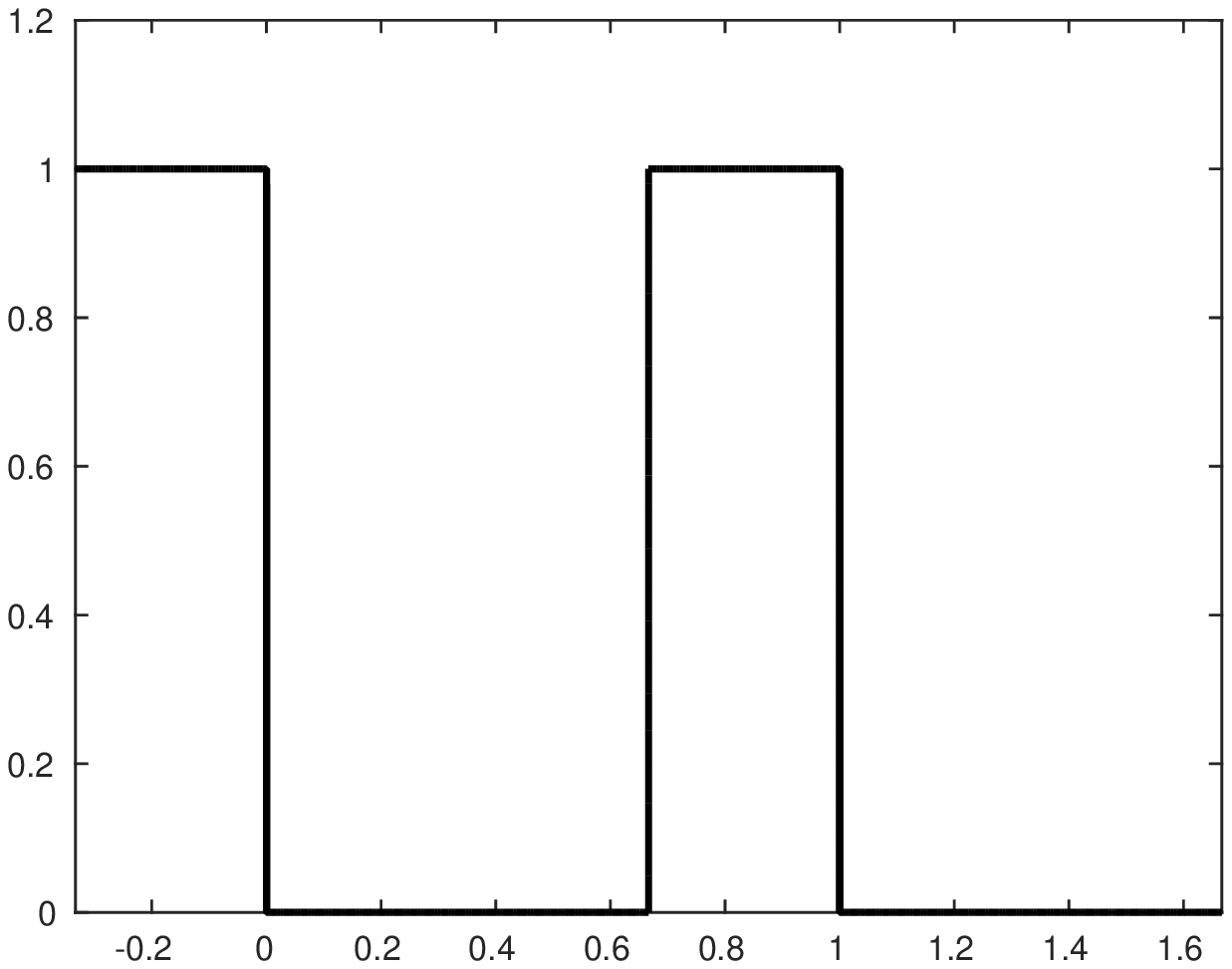}
	\end{minipage}
	\begin{minipage}{.4\textwidth}
		\includegraphics[scale=0.35]{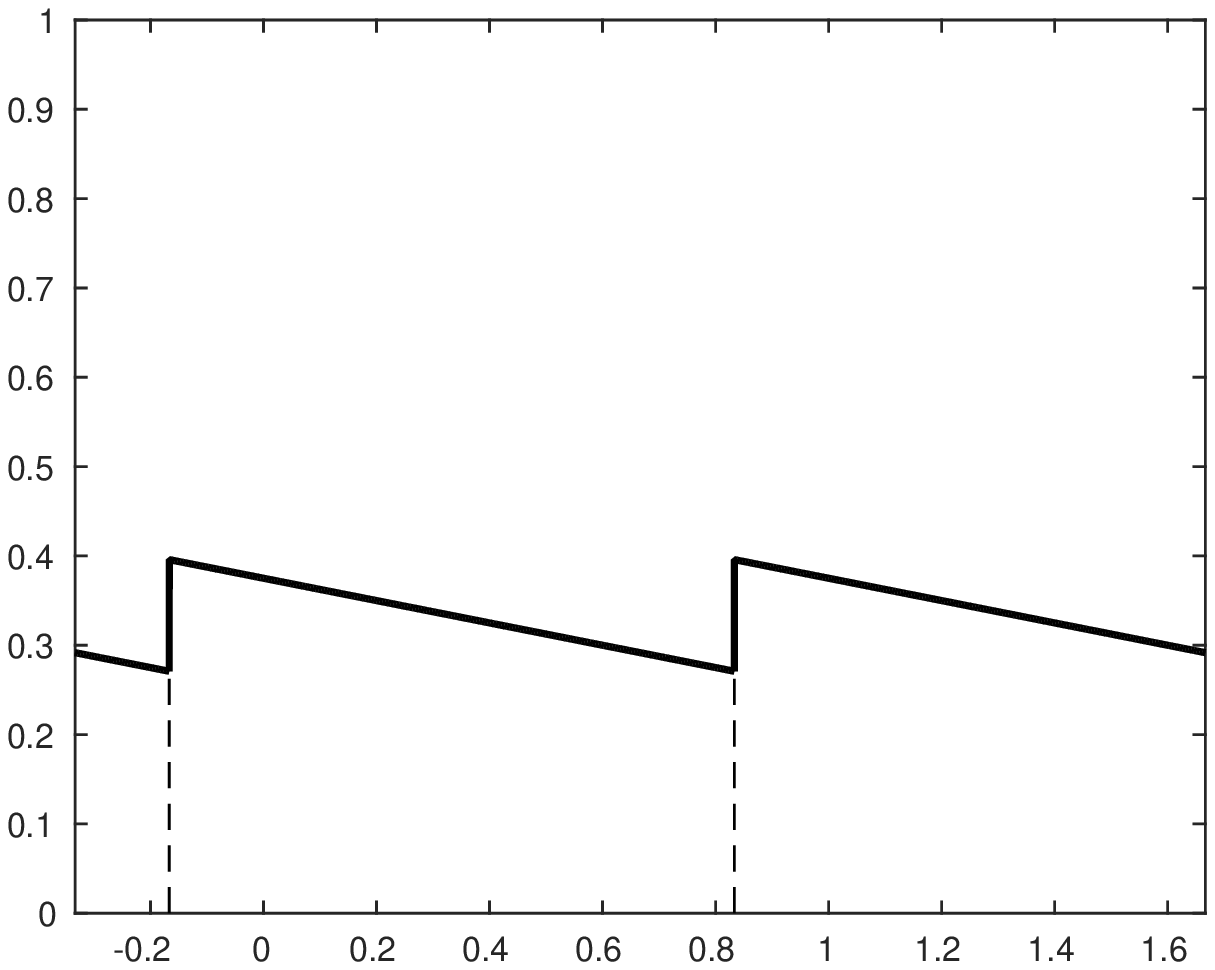}
	\end{minipage}\\
	%
	\begin{minipage}{.4\textwidth}
		\includegraphics[scale=0.35]{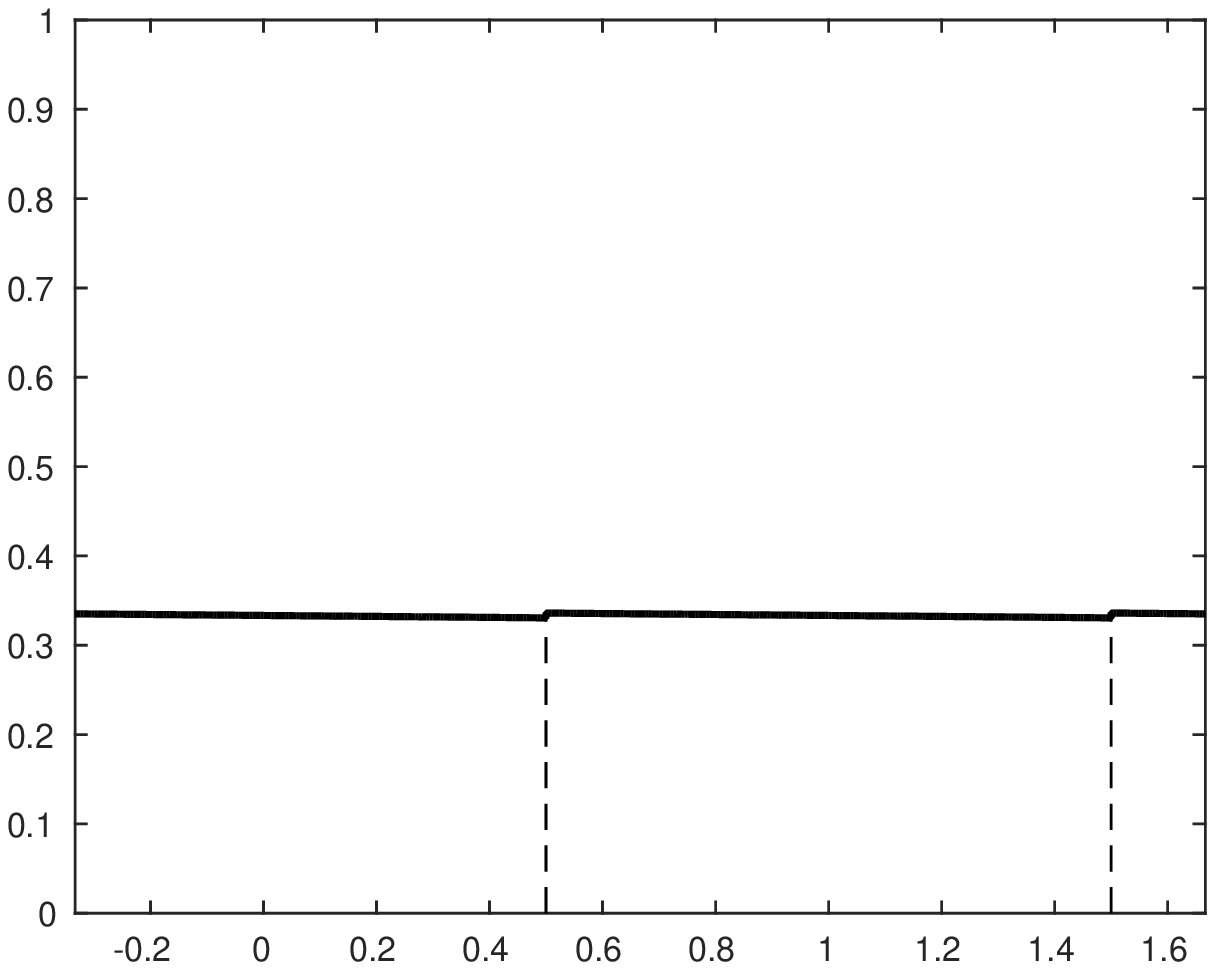}
	\end{minipage}
	\begin{minipage}{.4\textwidth}
		\includegraphics[scale=0.35]{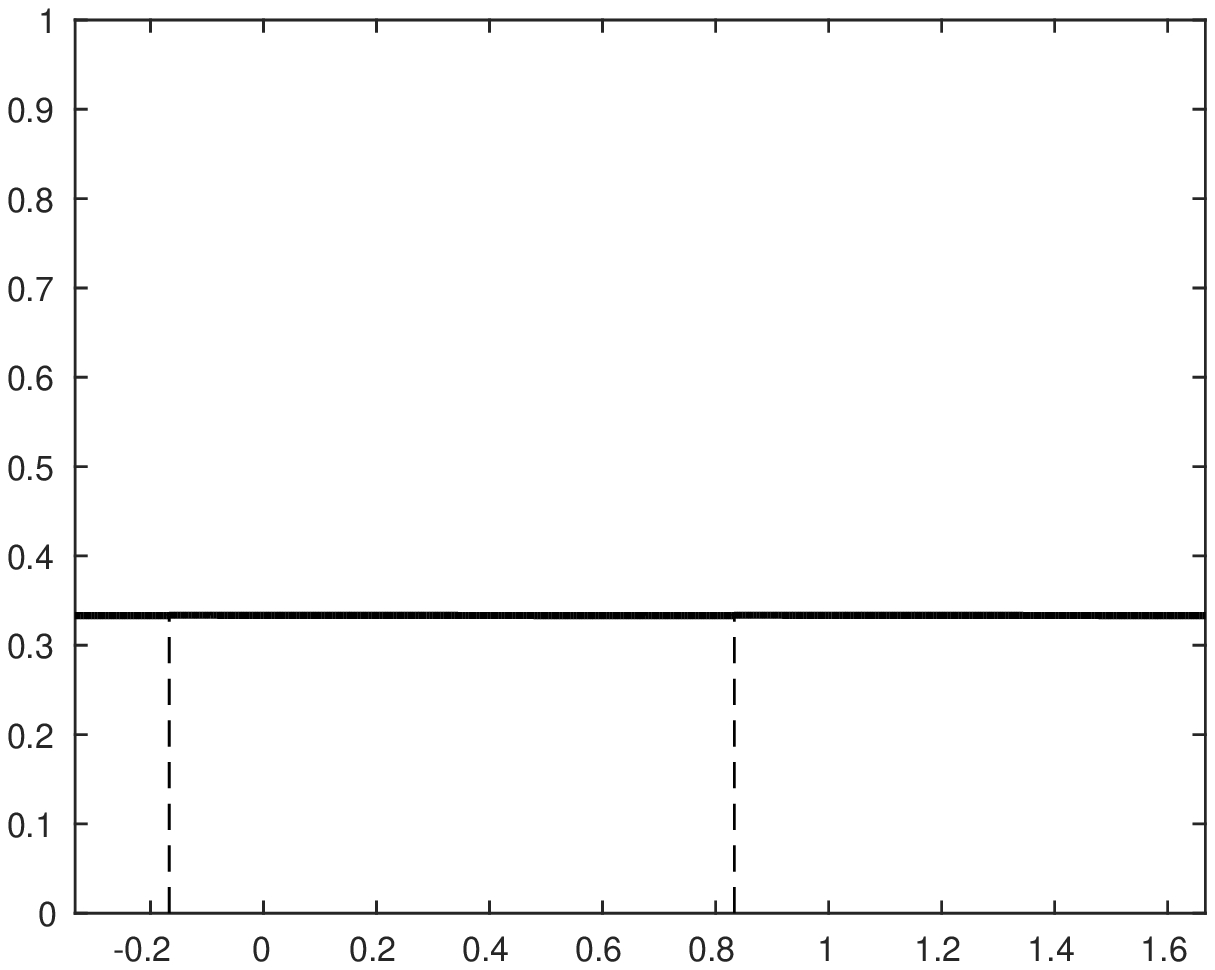}
	\end{minipage}
	\caption{The evolution of density profile for periodic step initial condition when $\rho=1/3$. 
	The four figures correspond to the cases when $t/L=0$, $t/L=4$, $O(1)\ll t/L\ll L^{1/2}$ (with the choice of   $t/L=93$ and $L=10^6$), 
		and $t=O(L^{3/2})$ (with the choice of $t=L^{3/2}$ and $L=10^6$). 
		}
	\label{fig:densityprofile}
\end{figure}


For comparison, we consider the infinite TASEP with same periodic step initial condition. 
Note that in this case the system depends on $L$ and $\rho$, unlike the flat initial condition which depends only on $\rho$. 
The interesting feature about the periodic step initial condition is that it generates (infinitely many, spatially periodic) shocks. 
For the flat initial condition, the limit theorem for the infinite TASEP was previously established, but for the periodic step initial condition, we first need to establish the limit theorem for infinite TASEP. 
We compute the particle fluctuations explicitly when $t\ll L^{3/2}$ both away from the shocks and near the shocks, and then show that the periodic TASEP have the same fluctuations.

We state our result for the particles whose labels are close to $[\alpha N]$, where $\alpha\in\realR$ is a fixed constant.
By periodicity of the initial condition, we may assume $0< \alpha\le 1$.
The shocks, once generated, travel with average speed  $1-2\rho$ (see~\eqref{eq:shockloca} below). 
On the other hand,  the speed of a given particle changes with time and asymptotically   becomes $1-\rho$. 
Since the shocks and the particle have different speeds,  a particle meets with a shock once every $O(N)$ time interval. 
By solving the Burger's equation, one can check that (see Appendix~\ref{sec:appendix} below) the $[\alpha N]$-th particle meets with a shock at times $s_1N,s_2N,\cdots$ on average, 
where $s_j=s_j(\alpha)$ are defined by 
\begin{equation}
\label{eq:aux_019}
s_{j}=s_j(\alpha):=
\frac{(\sqrt{j-\alpha+1}+\sqrt{j-\alpha})^2}{4\rho^2}, \qquad j\ge 1, 
\end{equation}
for $0<\alpha\le 1$.
Now define the sequence of (scaled) shock intervals for particle $[\alpha N]$: 
\begin{equation}
\mathcal{S}_{j}=(s_{j},s_{{j+1}}), \qquad j=0,1,2,\cdots, 
\end{equation}
where we set 
\begin{equation}
\label{eq:aux_019005}
	s_{0}=s_0(\alpha):=1-\alpha. 
\end{equation}
Here $s_0N$ is not a shock time for particle $[\alpha N]$. Instead it is the average time the particle moves for the first ime; due to the initial condition, it takes certain amount of time for a given particle to be able to move. 

We introduce the notation $I^{(\epsilon)}=\{x: x\in I, \dist(x,\partial I) \ge \epsilon\}$ for interval $I$ and nonnegative number $\epsilon$. 
Note that if $I=(a,b)$, then $I^{(\epsilon)}=[a+\epsilon,b-\epsilon]$ for $0<\epsilon<(b-a)/2$, and $I^{(\epsilon)}=\emptyset$ for $\epsilon>(b-a)/2$. 
We denote by $\mathcal{A}_2(u)$ the Airy$_2$ process \cite{Prahofer-Spohn02} whose marginals have the GUE Tracy-Widom distribution \cite{Tracy-Widom94}.

As in the flat case, we state the results for a sequence of times $t_N$. 
We introduce another sequence of parameters $j_N$. 
The parameter $j_N$ measures the number of encounters with shocks by time $t_N$. 
Since we assume sub-relaxation time scale $t_N\le O(N^{3/2-\epsilon})$, 
the parameter $j_N$ satisfies $j_N\le O(N^{1/2-\epsilon})$. 
We state the result for two cases separately; when the particles are away from shocks and when they are near a shock. 
These two cases correspond to the cases when the rescaled time parameter $\ttt_N= \frac{t_N}{N}$ 
satisfies $\ttt_N\in \mathcal{S}_{j_N}^{(\epsilon')}$ and when $\ttt_N= s_{j_N}$, respectively. 
Note that $j_N/\ttt_N=O(1)$. 


\subsubsection{Away from shocks}

We state the first theorem only for $0<\alpha<1$. The case $\alpha=1$ is discussed in Remark~\ref{rmk:alphaze}.

\begin{thm} (Periodic step initial condition 1. Away from shocks)
	\label{thm:limiting_process_step}
\begin{enumerate}[(i)]
\item
	Let $\rho$ and $\alpha$ be fixed constants satisfying $0<\rho\le 1/2$ and $0< \alpha< 1$. 
	Consider  the infinite TASEP with the periodic step initial condition~\eqref{eq:step_ic} with $L=[\rho^{-1} N]$.
		For fixed constants $\epsilon>0$ and $C>0$, 
		let $j_N$ be  an integer sequence satisfying $0\le j_N<CN^{1/2-\epsilon}$ for all $N$.
		Fix $\epsilon'>0$ such that 
		$\mathcal{S}_{j_N}^{(\epsilon')}\ne \emptyset$ for all $N$. 
		Set
		\begin{equation}
		\label{eq:aux_004005}
		\mu=\mu(N) 
		:=\sqrt{\frac{j_N-\alpha+1}{\ttt_N}}
		\end{equation}
		and
		\begin{equation}
		\label{eq:params}
		\params=\params(N):= 2\mu^{4/3}(1-\mu)^{1/3},\qquad \sigma_2=\sigma_2(N):=\mu^{-1/3}(1-\mu)^{2/3}.
		\end{equation}
		Defined the scaled particle location as
		\begin{equation}
		\label{eq:rescaled_process_step}
			\X_N(u):= x_{\left[\alpha N+\params ut_N^{2/3}\right]}(t_N)-(1-2\mu)t_N-\rho^{-1}j_NN.
		\end{equation}
		Then for every time sequence $t_N=\ttt_N N$ satisfying $\ttt_N\in \mathcal{S}_{j_N}^{(\epsilon')}$, we have
		\begin{equation}
		\label{eq:limit_step_001}
		\frac{\X_N(u)-2\mu^{1/3}(1-\mu)^{1/3}ut_N^{2/3}}{-\sigma_2t_N^{1/3}}\longrightarrow \mathcal{A}_2(u)-u^2		
		\end{equation}
		for $u\in\realR$ in the sense of convergence of finite dimensional distribution as $N\to \infty$.
\item The same result holds for the periodic TASEP of period $L=[\rho^{-1}N]$ with $N$ particles.
\end{enumerate}
\end{thm}

As mentioned before, due to the periodic step initial condition, the $[\alpha N]$-th particle only starts to move after $s_0N$ time on average. 
This is the reason that we restrict that $t_N\ge s_0N$: see the condition that $j_N\ge 0$. 
For the next theorem about the fluctuations near a shock, we assume that $j_N\ge 1$ so that $t_N\ge s_1 N$. This is because $s_1N$ is the average of the first time for the $[\alpha N]$-th particle to meet a shock. 

\begin{rmk}\label{rmk:alphaze}
We now discuss the case of when $\alpha=1$. 
This case is concerned with the particles labeled $N+O(1)$. 
At time $0$, the particle labeled $N$ has the particle $N-1$ to its immediate left but the particle $N+1$ is in $O(N)$ distance to its right. 
The distance between particle $N$ and particle $N+1$ becomes $O(1)$ at around time $s_1(1)N$, the first shock time for $\alpha=1$. 
We can show that if we add the assumption that $j_N\ge 1$ for all large enough $N$, then the theorem holds for $\alpha=1$. 
\end{rmk}

\subsubsection{Near shocks}

Before we discuss the next theorem, let us consider the density profile. 
It is same for both infinite TASEP and periodic TASEP. 
A computation of the Burger's equation shows  that the density profile  at a shock has jump discontinuity of order $O(Lt^{-1})$. See Appendix~\ref{sec:appendix} for the computation and Figure~\ref{fig:densityprofile} for an illustration.
When $t\le O(L)$, the discontinuity is order $O(1)$. 
On the other hand, when $t\gg L$, the discontinuity converges to zero and the density profile is asymptotically continuous. 
Indeed the density profile converges to the constant function with value $\rho$. 
However, in the sub-relaxation time scale $t\le O(L^{3/2-\epsilon})$, the discontinuity is at least $O(L^{\epsilon-1/2})$. 
The next theorem shows that as long as the discontinuity is at least this much, the fluctuations near a shock are same as the $O(1)$ discontinuity, and, furthermore, the periodic TASEP has the same fluctuations as the infinite TASEP. 



The part (c) of the next theorem is about the fluctuations at a shock. 
Recently, Ferrari and Nejjar \cite{Ferrari-Nejjar15} studied them for infinite TASEP. They obtained a simple general result from which they computed the fluctuations for a few examples of deterministic initial conditions. Especially, when the density profile has a discontinuity of order $O(1)$ at the shock and both sides of the shock have Airy$_2$ fluctuations, the particle fluctuations at the shock are distributed as the maximum of two independent GUE Tracy-Widom random variables with possibly different variances. 
The part (i) (c) of the next theorem shows that the infinite TASEP with periodic step initial condition has the same property. 
However, there are two main differences from \cite{Ferrari-Nejjar15}. 
The first is that there are infinitely many shocks here whereas there is only one shock in \cite{Ferrari-Nejjar15}. 
The other is that the discontinuity can be as small as $O(L^{\epsilon-1/2})$. 
These differences make the analysis more complicated. 

\begin{thm} (Periodic step initial condition 2. Near a shock)
	\label{thm:limiting_process_step_shock}
\begin{enumerate}[(i)]
\item
	Consider the infinite TASEP as in the previous theorem with $0< \alpha\le 1$. For fixed constants $\epsilon>0$ and $C>0$, 
		let $j_N$ be  a  positive integer sequence satisfying $1\le  j_N<CN^{1/2-\epsilon}$ for all $N$. 
		Then for the shock time sequence $t_N=s_{j_N} N$, we have the following result.
		Set
		\begin{equation}
		\label{eq:aux_004005shock}
		\mu=\mu(N) 
		:=\sqrt{\frac{j_N-\alpha+1}{s_{j_N}}}.
		\end{equation}
		(This is same as~\eqref{eq:aux_004005} with $\ttt_N$ replaced by $s_{j_N}$). 
		Let $\params$, $\sigma_2$ be given by~\eqref{eq:params} and $\X_N(u)$ by~\eqref{eq:rescaled_process_step} with the new definition of $\mu$.
	\begin{enumerate}[(a)]
	\item (Particles in the higher density profile; $u>0$) We have 
		\begin{equation}
		\label{eq:limit_step_001shock}
		\frac{\X_N(u)-2\mu^{1/3}(1-\mu)^{1/3}ut_N^{2/3}}{-\sigma_2t_N^{1/3}}\longrightarrow \mathcal{A}_2(u)-u^2
		\end{equation}
		for $u>0$ in the sense of convergence of finite dimensional distribution as $N\to \infty$.
		
	\item (Particles in the lower density profile; $u<0$)
		We have 
		\begin{equation}		
		\label{eq:limit_step_001shock2}
		\frac{\X_N(\ara u)-2\tilde\mu^{1/3}(1-\tilde\mu)^{1/3}ut_N^{2/3}}{-\arao\sigma_2 t_N^{1/3}}\longrightarrow  \mathcal{A}_2(u)-u^2
		\end{equation}
		for $u<0$ in the sense of convergence of finite dimensional distribution as $N\to\infty$, where
		\beq
		\label{eq:aux_2016_08_24_03}
			\tilde \mu =\tilde \mu(N) :=\sqrt{\frac{j_N-\alpha}{s_{j_N}}}, 
			\qquad 
			\ara=\ara(N) := \frac{\tilde\mu^{4/3}(1-\tilde\mu)^{1/3}}{\mu^{4/3}(1-\mu)^{1/3}},
			\qquad
			\arao=\arao(N):=\frac{\tilde{\mu}^{-1/3}(1-\tilde\mu)^{2/3}}{\mu^{-1/3}(1-\mu)^{2/3}}.
		\eeq
	\item (Particle at the discontinuity point; $u=0$) If we further assume $\displaystyle \lim_{N\to\infty}j_N$ exists in $[1, \infty]$, then 
		\begin{equation}
		\label{eq:limit_step_at_shock}
		\frac{\X_N(0)}{-\sigma_2t_N^{1/3}}\Rightarrow \max\{\chi_{2}^{(1)}, r \chi_{2}^{(2)} \}
		\end{equation}
		in distribution as $N\to\infty$, where 
		$\chi_2^{(k)}$, $k=1,2$, are two independent GUE Tracy-Widom random variables. 
		The parameter $r=\lim_{N\to\infty}\arao(N)$ is given by 
		\begin{equation}
		r= \left(\frac{j_\infty-\alpha}{j_\infty-\alpha+1}\right)^{-1/6}\left(\frac{\sqrt{s_{j_\infty}}-\sqrt{{j_\infty}-\alpha}}{\sqrt{s_{j_\infty}}-\sqrt{{j_\infty}-\alpha+1}}\right)^{2/3}
		\end{equation}
		if  $j_\infty:=\displaystyle \lim_{N\to\infty}j_N$ is finite, and $r=1$ if $\displaystyle \lim_{N\to\infty}j_N=\infty$. 
	\end{enumerate}
\item The same result holds for the periodic TASEP. 
\end{enumerate}
\end{thm}

The Airy$_2$ processes appearing in parts (a) and (b) are independent. 
The two cases (a) and (b) can be formally expressed as a uniform formula which holds for all $u\in\realR$ as follows.
Let $\mathcal{A}_2(u)$ and $ \tilde{\mathcal{A}}_2(u)$ be two independent Airy$_2$ processes, and set 
$\mathcal{A}^{(1)}(u)= \mathcal{A}_2(u)-u^2$ and $\mathcal{A}^{(2)}= \tilde{\mathcal{A}}_2(u)-u^2$. 
The result~\eqref{eq:limit_step_001shock} for case (a) can be written formally as
\beq \label{eq:tempor0}
	\frac{\X_N(u)}{- \sigma_2 t_N^{1/3}}\approx \aras ut_N^{1/3} + \mathcal{A}^{(1)}(u), 
	\qquad u>0,
\eeq
where $\aras:=-2\mu^{2/3}(1-\mu)^{-1/3}$.
On the other hand, if we replace $u$ by $u/\ara$ in case (b),~\eqref{eq:limit_step_001shock2} can be written formally as
\beq \label{eq:tempor1}
	\frac{\X_N(u)}{-\sigma_2 t_N^{1/3}}\approx \arat ut_N^{1/3}+\arao \mathcal{A}^{(2)}(\ara^{-1} u), \qquad u<0, 
\eeq
where $\arat:=-2\tilde{\mu}^{-1}\mu^{5/2}(1-\mu)^{-1/3}$. Note that $\arat=\frac{\mu}{\tilde{\mu}}\aras<\aras<0$. Theorem~\ref{thm:limiting_process_step_shock} (a) and (b) can thus formally be written as 
\begin{equation}
\label{eq:tempor11}
	\frac{\X_N(u)}{-\sigma_2 t_N^{1/3}}\approx 
	\max	\left\{\aras ut_N^{1/3}+	\mathcal{A}^{(1)}(u), \, \arat ut_N^{1/3}+\arao \mathcal{A}^{(2)}(\ara^{-1} u) \right\} . 
\end{equation}
The leading order term on the right-hand side is $t_N^{1/3}$. 
Since $\arat<\aras<0$, we find that the maximum is different for $u>0$ and $u<0$ resulting in~\eqref{eq:tempor0} and~\eqref{eq:tempor1}, respectively. 

Formally,~\eqref{eq:tempor11} implies (c). Indeed, setting $u=0$ in~\eqref{eq:tempor11} and recalling that the marginals of the Airy$_2$ process are distributed as the GUE Tracy-Widom distribution, we formally obtain 
\beq \label{eq:tempor1112}
	\frac{\X_N(0)}{-\sigma_2 t_N^{1/3}}\approx 
	\max\{  \chi_2^{(1)}, \arao  \chi_2^{(2)} \}
	\approx 	\max\{  \chi_2^{(1)}, r  \chi_2^{(2)} \}.
\eeq
We may also set $u=\xi t_N^{-1/3}$ in~\eqref{eq:tempor11} and formally obtain 
\beq \label{eq:tempor12}
    \frac{\X_N(\xi t_N^{-1/3})}{-\sigma_2 t_N^{1/3}}\approx 
	\max\{ \aras \xi + \chi_2^{(1)}, 
	\arat \xi+\arao \chi_2^{(2)} \}
	\approx \max\{ c_1 \xi + \chi_2^{(1)}, 
	c_2 \xi+r \chi_2^{(2)} \}
\eeq
for all $\xi\in\realR$, where $c_1,c_2$ are the limits of $\aras$ and $\arat$ as $N\to\infty$. Note that the GUE Tracy-Widom random variables $\chi_2^{(1)}, \chi_2^{(2)}$ are same for all $\xi$. 
A slight extension of the proof for the case (c) implies this result but we do not include the proof in this paper. (See Corollary 2.7 of \cite{Ferrari-Nejjar15} for a similar case for infinite TASEP.) 

\begin{rmk} \label{rmk:rhob2}
When $\rho>1/2$, the following holds. In this case it is easy to see that, after time $(\rho^{-1}-1)N+O(N^{1/2})$, the rightmost particle will meet the leftmost one (on the ring) and then cannot move for certain time, which we call frozen time, until the local density becomes less than $1$. If a particle falls into its frozen time, it stays  at its location and the average speed is zero. Besides of this situation, the particle fluctuations are the same as that in the case $\rho\le 1/2$. More explicitly, let us define 
\begin{equation}
s'_{j}= \begin{dcases}
j+1-\alpha, & j=0,1,\cdots, j_0-1,\\
s_j,& j\ge j_0,
\end{dcases}
\end{equation}
and
\begin{equation}
s''_j=\begin{dcases}
\left(\sqrt{j-\alpha}+\sqrt{j(\rho^{-1}-1)}\right)^2, & j=1,\cdots, j_0-1,\\
s_j,& j\ge j_0,
\end{dcases}
\end{equation}
where
\begin{equation}
j_0:=\min\left\{j\ge 1; j\left(\sqrt{j-\alpha}+\sqrt{j+1-\alpha}\right)^2\ge \frac{\rho}{1-\rho}\right\}.
\end{equation}
Then Theorem \ref{thm:limiting_process_step} holds if we replace $\mathcal{S}_j$ by $(s'_j,s''_{j+1})$, and Theorem \ref{thm:limiting_process_step_shock} hold provided $j\ge j_0$ for all $j$. Note that when $j<j_0$, $s''_{j}<s'_{j}$, and hence there is a gap between the two intervals $\mathcal{S}_{j-1}$ and $\mathcal{S}_j$ which is exact a frozen time. 
\end{rmk}

\subsubsection{Relaxation and super-relaxation time scale}\label{sec:reldis}


In the super-relaxation time scale $t\gg L^{3/2}$, the density profile is flat for both infinite TASEP and periodic TASEP with periodic step initial condition. 
However, the two models are expected to have different fluctuations. 
For the infinite TASEP, the distance scale $O(L)$ of the initial condition is smaller than the spatial correlation scale $O(t^{2/3})$ when $t\gg L^{3/2}$, the initial condition is effectively flat. 
This can be seen easily using the corresponding directed last passage percolation: see Section~\ref{sec:inftaseprel} for a further discussion. 
Since the initial condition is still deterministic, we expect that the fluctuations are of order $t^{1/3}$ and are given by the Airy$_1$ process. 
On the other hand,  the periodic TASEP is expected to have $t^{1/2}$ fluctuations with Gaussian distribution 
because it is expected to be in the equilibrium dynamics due to the system size effect. 

In the relaxation time scale $t=O(L^{3/2})$, it was shown in  \cite{Prolhac16, Baik-Liu16} that the one-point distribution for the periodic TASEP has $t^{1/3}$ height fluctuations and converges to a distribution which is different from the GUE Tracy-Widom distribution. 
The infinite TASEP should also have $t^{1/3}$ height fluctuations but with the one-point distribution which is presumably different from the periodic case.
See Section~\ref{sec:inftaseprel} for a heuristic discussion about this distribution function.

\subsection{Organization of the paper}

We prove the theorems by studying the corresponding directed last passage percolation (DLPP) models: the periodic DLPP and the usual DLPP. 
In Section \ref{sec:estimates_DLPP_periodic_DLPP}, we estimate the probability of the event that the maximal path deviates from the diagonal path in both DLPP models. 
Translated into TASEP, this implies that in the sub-relaxation time scale, the fluctuations of the periodic TASEP and the infinite TASEP have the same distribution with high probability. 
Putting together with the known fluctuation results on the usual DLPP and their extensions to the periodic step initial condition, we prove the theorems in  Section \ref{sec:proof}. 
Some technical lemmas are postponed to Section \ref{sec:others}.
Finally, in Appendix~\ref{sec:appendix} we discuss the evolution of the density profile by solving the Burger's equation with the  periodic step initial condition.

\subsubsection*{Acknowledgments}
We would like to thank Ivan Corwin  and Patrik Ferrari for useful conversations and comments. 
The work of Jinho Baik was supported in part by NSF grants DMS1361782.

\section{Periodic directed last passage percolation} 
\label{sec:estimates_DLPP_periodic_DLPP}


The periodic directed last passage percolation (DLPP) model is defined as follows. 
The period of the model is a lattice point $\mr{v}=(\mr{v}_1,\mr{v}_2)\in\intZ^2$ satisfying $\mr{v}_2<0<\mr{v}_1$. 
We assign periodic random variables $w(\mr p)$ to $\mr p\in\intZ^2$ satisfying $w(\mr p)=w(\mr p+\mr v)$. 
Apart from the periodicity, we assume that the random variables $w(\mr p)$ are \iid
 The point-to-point last passage time from $\mr p=(\mr p_1,\mr p_2)\in\intZ^2$ to $\mr q=(\mr q_1,\mr q_2)\in\intZ^2$ is defined 
by 
\begin{equation}
\label{eq:aux_017}
H_{\mr p}(\mr q)=\begin{dcases}
\max_{\pi} \sum_{\mr r\in\pi}w(\mr r),& \text{if $\mr p_1\le \mr q_1$ and $\mr p_2\le \mr q_2$,}\\
-\infty,							  & \text{otherwise,} 
\end{dcases}
\end{equation}
where the maximum is taken among all the possible up/right lattice paths\footnote{A lattice path is a path consists of unit horizontal/vertical line segments whose endpoints are lattice points.} $\pi$  starting from $\mr p$ and ending at $\mr q$, and the summation is taken over all lattice points $\mr r$ on the path $\pi$. For simplification, we write $H_{\mr 0}(\mr q)=H(\mr q)$ when $\mr p=\mr 0$.


The usual DLPP model is the case with\footnote{$|\mr v|=\sqrt{\mr v_1^2+\mr v_2^2}$ denotes the norm of $\mr v=(\mr v_1, \mr v_2)$.} $|\mr v|=\infty$. In this case, $w(\mr p)$ are \iid for all $\mr p\in \intZ^2$. 
The  point-to-point last passage time for the usual DLPP model is denoted by $G_{\mr p}(\mr q)$. 
As above, we write $G_{\mr 0}(\mr q)$ by $G(\mr q)$. 
 

We assume that $w(\mr p)$ are exponential random variables with parameter $1$. 
The analysis in this paper can also be applied to geometric random variables but we do not discuss it here. 
The goal of this section is to compare the periodic DLPP with the usual DLPP. 
This is obtained by studying the transversal fluctuations. 


\subsection{DLPP} 
\label{sec:tail_DLPP}

We first study the usual DLPP model without the periodicity condition. 

\subsubsection{Tail estimates of point-to-point last passage time}
For any fixed positive constants $c_1<c_2$ we define
\begin{equation}\label{eq:Qdez}
\mr Q(c_1,c_2)=\left\{\mr q=(\mr q_1,\mr q_2)\in\intZ^2_{+}; c_1<\mr q_2/\mr q_1<c_2 \right\}.
\end{equation}
It is well known that \cite{Johansson00}  for fixed positive constants $c_1,c_2$ satisfying $c_1<c_2$, 
\begin{equation}
\label{eq:aux_042}
	\prob\left( \frac{G(\mr q)-d(\mr q)}{s(\mr q)} \le x\right) \to \FGUE(x) 
\end{equation}
for all $x\in \realR$ as $\mr q\in \mr Q(c_1,c_2)$ satisfies $|\mr q|\to \infty$.
The constants in~\eqref{eq:aux_042} are given by 
\begin{equation}
\label{eq:aux_026}
d(\mr q)=(\sqrt{\mr q_1}+\sqrt{\mr q_2})^2, \qquad s(\mr q)= (\mr q_1\mr q_2)^{-1/6}(\sqrt{\mr q_1}+\sqrt{\mr q_2})^{4/3}.
\end{equation}


The tail estimates of $G(\mr q)$ can be found in, for example, \cite[Section 3.1]{Baik-Ben_Arous-Peche05} and \cite[Sections 3 and 4]{Baik-Ferrari-Peche14}. 
The following estimates are slightly stronger than those written explicitly in the above references, but they can be obtained from the same analysis. We do not provide the detail here. 

\begin{lm}[Tail estimates for DLPP]
\label{prop:tail_estimate}
Suppose $c_1, c_2$ are both fixed positive constants satisfying $c_1<c_2$. Then there exist positive constants $x_0, C$, and $c$ such that
\begin{equation}
\label{eq:lower_tail}
\prob\left(\frac{G(\mr q)-d(\mr q)}{s(\mr q)}\ge -x\right)\ge 1-e^{-cx^{3/2}},
\end{equation}
and
\begin{equation}
\label{eq:upper_tail}
\prob\left(\frac{G(\mr q)-d(\mr q)}{s(\mr q)}\le x\right)\ge 1-e^{-cx},
\end{equation}
for all $x\ge x_0$ and $\mr q\in \mr Q(c_1,c_2)$ satisfying $|\mr q|\ge C$.
\end{lm}

\subsubsection{Transversal fluctuations of DLPP}

The maximal  path from $\mr p$ to $\mr q$ is concentrated about the diagonal line segment $\overline{\mr p\mr q}$ 
with the traversal fluctuations of order $|\mr q- \mr p|^{2/3}$ \cite{Johansson00a, Baik-Deift-McLaughlin-Miller-Zhou01}.\footnote{These results are for the Poissonian version of DLPP and the geometric random variables, but the results extend to exponential random variables. But we do not need these results. Instead we use the ideas in them to prove  Proposition~\ref{prop:tail_transversal} below which implies this statement for exponential variables. 
}  
For our purpose, we need to estimate the deviations from the diagonal line segment. 
For the Poisson version of the DLPP model, such an estimate was proved for the case when $\mr q_1=\mr q_2$ by 
Basu, Sidoravicius, and Sly in \cite{Basu-Sidoravicius-Sly16}. 
In our case, we need an uniform estimate for different end points $\mr q$ in $\mr Q(c_1,c_2)$ for the DLPP model with exponential random variables. 
See Remark~\ref{rmk:compwBSS} below for a comparison between our proof with that of \cite{Basu-Sidoravicius-Sly16}. 

We introduce two notations. 
For two lattice points $\mr p=(\mr p_1,\mr p_2)$ and $\mr q=(\mr q_1,\mr q_2)$ satisfying $\mr p_1\le \mr q_1$ and $\mr p_2\le \mr q_2$, 
we denote by $\pi^{max}_{\mr  p}(\mr q)$ a maximal up/right lattice path from $\mr p$ to $\mr q$: it satisfies $\sum_{\mr r\in \pi^{max}_{\mr  p}(\mr q)} w(\mr r)=G_{\mr p}(\mr q)$. 
If there are more than one such path, we pick the topmost one. 
When $\mr p=\mr 0$, we write $\pi^{max}(\mr q)$ for  $\pi^{max}_{\mr  0}(\mr q)$. 
For $y>0$, we set $B_{\overline{\mr p\mr q}}(y)=\{\mr r\in\intZ^2; \dist ( \mr r, \overline{\mr p\mr q} )\le y\}$.

\begin{prop}[Estimate for the transversal fluctuations of DLPP]
\label{prop:tail_transversal}
For fixed positive constants $c_1 < c_2$ and  $\epsilon$,  there exist positive constants $C$ and $c$ such that
\begin{equation}
\label{eq:tail_transversal}
\prob\left(\pi^{max}(\mr q)\subseteq B_{\overline{\mr 0\mr q}}(y|\mr q|^{2/3})\right)\ge 1- e^{-cy^2}
\end{equation}
for all $\mr q\in \mr Q(c_1,c_2)$ satisfying $|\mr q|>C$, and all $y$ satisfying $(\log|\mr q|)^{1/2+\epsilon}\le y \le |\mr q|^{1/3}$.
\end{prop}


\begin{rmk}\label{rmk:compwBSS}
For the Poissonian version of the DLPP, Basu, Sidoravicius, and Sly obtained a weaker lower bound $1-e^{-cy}$ for the special case $\mr q_1=\mr q_2$ only, but their estimate applies to all $y\ge C$.  
See \cite[Theorem 11.1]{Basu-Sidoravicius-Sly16}. 
They used a tail estimate on where the maximal path intersects the middle vertical line $\{\mr p=(\mr p_1,\mr p_2)\in\realR^2; \, \mr p_1=\mr q_1/2\}$, and then used a so-called chaining argument to recurrently apply this estimate. 
Our approach is different. 
We mainly use the tail estimates in Lemma~\ref{prop:tail_estimate} and prove that the probability the maximal path intersects the boundary of $B_{\overline{\mr 0\mr q}}(y|\mr q|^{2/3})$ is very small.
The proof given here is much simpler since we assume that $y\ge (\log|\mr q|)^{1/2+\epsilon}$. 
Basu, Sidoravicius, and Sly needed more delicate argument in order to obtain an estimate for $y=O(1)$.
On the other hand, here we need to be careful with the uniformity in $\mr q$ in $\mr Q(c_1,c_2)$. 
  
We also note that the Gaussian estimate is not optimal. It is believed that the optimal exponent in the tail should be $y^3$ since the argmax of $\mathcal{A}_2(y)-y^2$ (where $\mathcal{A}_2(y)$ is the Airy$_2$ process) has the tail $e^{-c|y|^3}$, see \cite{Corwin-Hammond14,Quastel-Remenik15,Bothner-Liechty13} for the discussions about the argmax.

\end{rmk}

\begin{proof}[Proof of Proposition~\ref{prop:tail_transversal}]
Fix $\mr q\in \mr Q(c_1,c_2)$.
We set $\Gamma= \overline{\mr 0\mr q}$. 
Consider the two lines which are parallel to $\overline{\mr 0\mr q}$ and of distance $y|\mr q|^{2/3}$ to $\overline{\mr 0\mr q}$. We denote by $\Gamma^{+}$ the part of the top line which lies inside the rectangle with vertices $(0,0), (\mr q_1, 0), (\mr q_1, \mr q_2)$, and $(0, \mr q_2)$. 
The similar part of the bottom line is denoted by $\Gamma^{-}$. 
See Figure~\ref{fig:strip_DLPP_1}.
\begin{figure}
\centering
\includegraphics[scale=0.25]{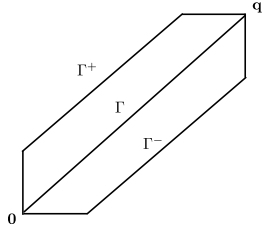}
\caption{Illustration of $\Gamma$ and $\Gamma^\pm$}
\label{fig:strip_DLPP_1}
\end{figure}

We say that a lattice point $\mr p$ is \emph{neighboring} to line $\mathcal{L}$ if there exists a real constant $c$ with $|c|<1$ satisfying $\mr p +c\mr e$ is on $\mathcal{L}$, where $\mr e$ is either the vector $(0,1)$ or $(1,0)$. Note that when $c=0$, $\mr p$ is on the line $\mathcal{L}$. We also call the set of all lattice points neighboring to the line the \emph{neighborhood} of the line. One key fact about this definition is that any lattice path intersecting a line will contain at least one point in the neighborhood of the line. 

Set
\begin{equation}
A^\pm=\{\mr p=(\mr p_1,\mr p_2)\in\intZ^2 \, ; \, \mr p \mbox{ is neighboring to } \,   \Gamma^\pm, 0\le \mr p_1\le \mr q_1, \, 0\le \mr p_2\le \mr q_2\}.
\end{equation}
Since a lattice path intersecting 
 $\Gamma^\pm$ contains at least one lattice point in $A^\pm$, we have
\begin{equation}
\label{eq:aux_016}
\begin{split}
\prob\left(\pi^{max}(\mr q)\subseteq B_{\overline{\mr 0\mr q}}(y|\mr q|^{2/3})\right) 
	&= 1-\prob\left(\pi^{max}(\mr q) \mbox{ intersects } \Gamma^\pm\right)\\ 
	&\ge 1- \sum_{\mr p\in A^\pm}\prob\left(\mr p\in \pi^{max}(\mr q)\right).
\end{split}
\end{equation}
Note that there are only $O(|\mr q|)$ points in $A^\pm$. 
Since we assume that $y \ge (\log|\mr q|)^{1/2+\epsilon}$, the proposition follows if we show that 
\begin{equation}
\label{eq:aux_009}
\prob\left(\mr p\in \pi^{max}(\mr q)\right) \le e^{-cy^2}
\end{equation} 
for all $\mr p\in A^\pm$. 
By symmetry, we only consider $A^+$. 
We consider three cases separately. 

\medskip


Case 1: Suppose that the point $\mr p=(\mr p_1,\mr p_2)\in A^+$ satisfies $\mr p_1\le \epsilon_1y|\mr q|^{2/3}$. This implies that $\mr p$ is near the left endpoint of $\Gamma^+$.  Here $\epsilon_1$ is a small positive constant independent of $\mr q$. The value of $\epsilon_1$ will be  determined later.  

Let $\mr r=(\mr r_1,\mr r_2)$ be the point in $A^+$ satisfying 
\begin{equation}
\mr r_1 = \left[\epsilon_1 y|\mr q|^{2/3} \right], \qquad \mr r_2= \frac{y|\mr q|^{5/3}}{\mr q_1}+\epsilon_1\frac{y\mr q_2|\mr q|^{2/3}}{\mr q_1}+O(1)
\end{equation}
where $O(1)$ is a term bounded by $2$ so that $\mr p\in A^+$. We also denote by $\mr r'=(\mr r'_1,\mr r'_2)$ the point in $A^+$ with
\begin{equation}
\mr r'_1=0,\qquad \mr r'_2=\frac{y|\mr q|^{5/3}}{\mr q_1}+O(1).
\end{equation}
See Figure~\ref{fig:strip_DLPP_2}. 
We observe that in the event that $\mr p\in \pi^{max}(\mr q)$, we have 
$G(\mr q)= G(\mr p)+G_{\mr p}(\mr q)\le G(\mr r)+G_{\mr r'}(\mr q)$
since $G_{\mr p}(\mr q)\le G_{\mr r'}(\mr q)$ and $G(\mr p)\le G(\mr r)$.
Therefore, 
\begin{equation}
\label{eq:aux_010}
\prob\left(\mr p\in \pi^{max}(\mr q)\right) \le \prob\left(G(\mr q)\le G(\mr r)+G_{\mr r'}(\mr q)\right).
\end{equation}
Now we take $\epsilon_1$ sufficiently small such that
\begin{equation}
\label{eq:aux_008}
d(\mr r)+d(\mr q-\mr r')\le d(\mr q) -c_3y|\mr q|^{2/3},
\end{equation}
for some positive constant $c_3$. 
This is is equivalent to 
\begin{equation}
\left(\sqrt{\epsilon_1 y|\mr q|^{2/3}}+\sqrt{\frac{y|\mr q|^{5/3}}{\mr q_1}+\epsilon_1\frac{y\mr q_2|\mr q|^{2/3}}{\mr q_1}}\right)^2 +\left(\sqrt{\mr q_1}+\sqrt{\mr q_2-\frac{y|\mr q|^{5/3}}{\mr q_1}}\right)^2 +O(1)\le \left(\sqrt{\mr q_1}+\sqrt{\mr q_2}\right)^2-c_3y|\mr q|^{2/3}.
\end{equation}
The left hand side of the above inequality equals to
\begin{equation}
\left(\sqrt{\mr q_1}+\sqrt{\mr q_2}\right)^2- \left(\frac{|\mr q|}{\sqrt{\mr q_1\mr q_2}}-\epsilon_1\frac{\mr q_1+\mr q_2}{\mr q_1}-2\sqrt{\epsilon_1}\frac{\sqrt{|\mr q|+\epsilon_1\mr q_2}}{\sqrt{\mr q_1}}\right) \cdot y|\mr q|^{2/3}+o(y|\mr q|^{2/3}).
\end{equation}
Therefore, if we choose $c_3<\min_{\mr q\in\mr  Q(c_1,c_2)}\frac{|\mr q|}{\sqrt{\mr q_1\mr q_2}}$ and $\epsilon_1$ small enough such that
\begin{equation}
\epsilon_1\frac{\mr q_1+\mr q_2}{\mr q_1}+2\sqrt{\epsilon_1}\frac{\sqrt{|\mr q|+\epsilon_1\mr q_2}}{\sqrt{\mr q_1}}<\min_{\mr q\in \mr Q(c_1,c_2)}\frac{|\mr q|}{\sqrt{\mr q_1\mr q_2}} -c_3
\end{equation}
uniformly for all $\mr q\in \mr Q(c_1,c_2)$, then~\eqref{eq:aux_008} holds.

Note that $\mr r\in \mr Q(c_1',c_2')$ and $\mr q-\mr r'\in \mr Q(c_1',c_2')$ for some positive constants $c_1',c_2'$ which depends only  on $\epsilon_1$, $c_1$, and $c_2$. 
Hence using Lemma~\ref{prop:tail_estimate}, we obtain
\begin{equation}
\label{eq:aux_011}
\begin{split}
&\prob\left(G(\mr q)\le G(\mr r)+G_{\mr r'}(\mr q)\right)\\
\le & 1- \prob\left(G(\mr q)> d(\mr q) -\frac{c_3}{3}y|\mr q|^{2/3}, G(\mr r)< d(\mr r)+\frac{c_3}{3}y|\mr q|^{2/3}, G_{\mr r'}(\mr q)<d(\mr q-\mr r')+\frac{c_3}{3}y|\mr q|^{2/3}\right)\\
\le& 3-\prob\left(G(\mr q)> d(\mr q) -\frac{c_3}{3}y|\mr q|^{2/3}\right)-\prob\left(G(\mr r)< d(\mr r)+\frac{c_3}{3}y|\mr q|^{2/3}\right)-\prob\left(G_{\mr r'}(\mr q)<d(\mr q-\mr r')+\frac{c_3}{3}y|\mr q|^{2/3}\right)\\
\le& e^{-cy|\mr q|^{1/3}} \le e^{-cy^2}
\end{split}
\end{equation}
for some constant $c$ independent of $\mr q$. Together with~\eqref{eq:aux_010}, this implies~\eqref{eq:aux_009}.

\begin{figure}
\centering
\begin{minipage}{.4\textwidth}
\includegraphics[scale=0.25]{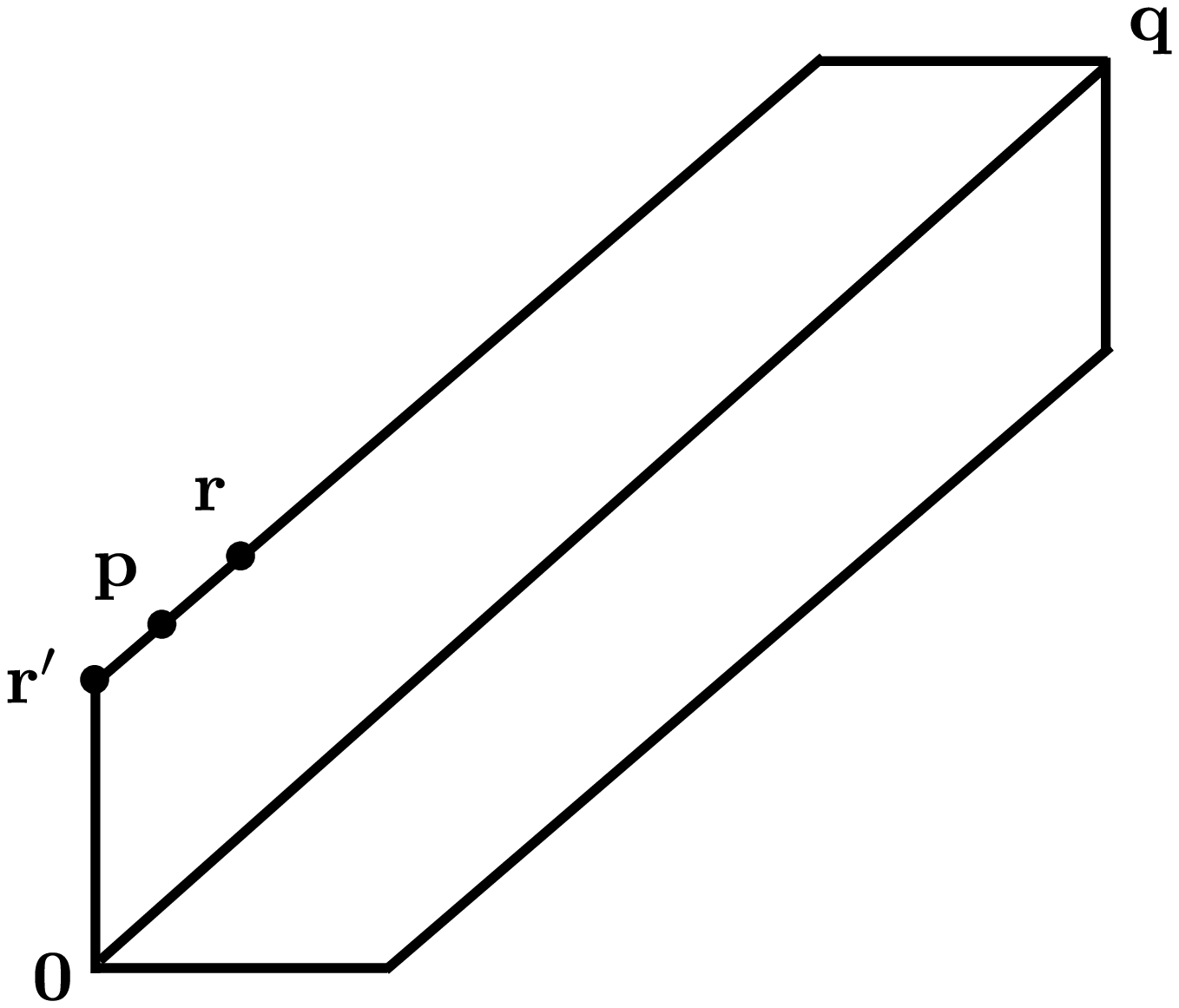}
\caption{Illustration of $\mr p$, $\mr r$ and $\mr r'$ in Case 1}
\label{fig:strip_DLPP_2}
\end{minipage}
\begin{minipage}{.4\textwidth}
\includegraphics[scale=0.25]{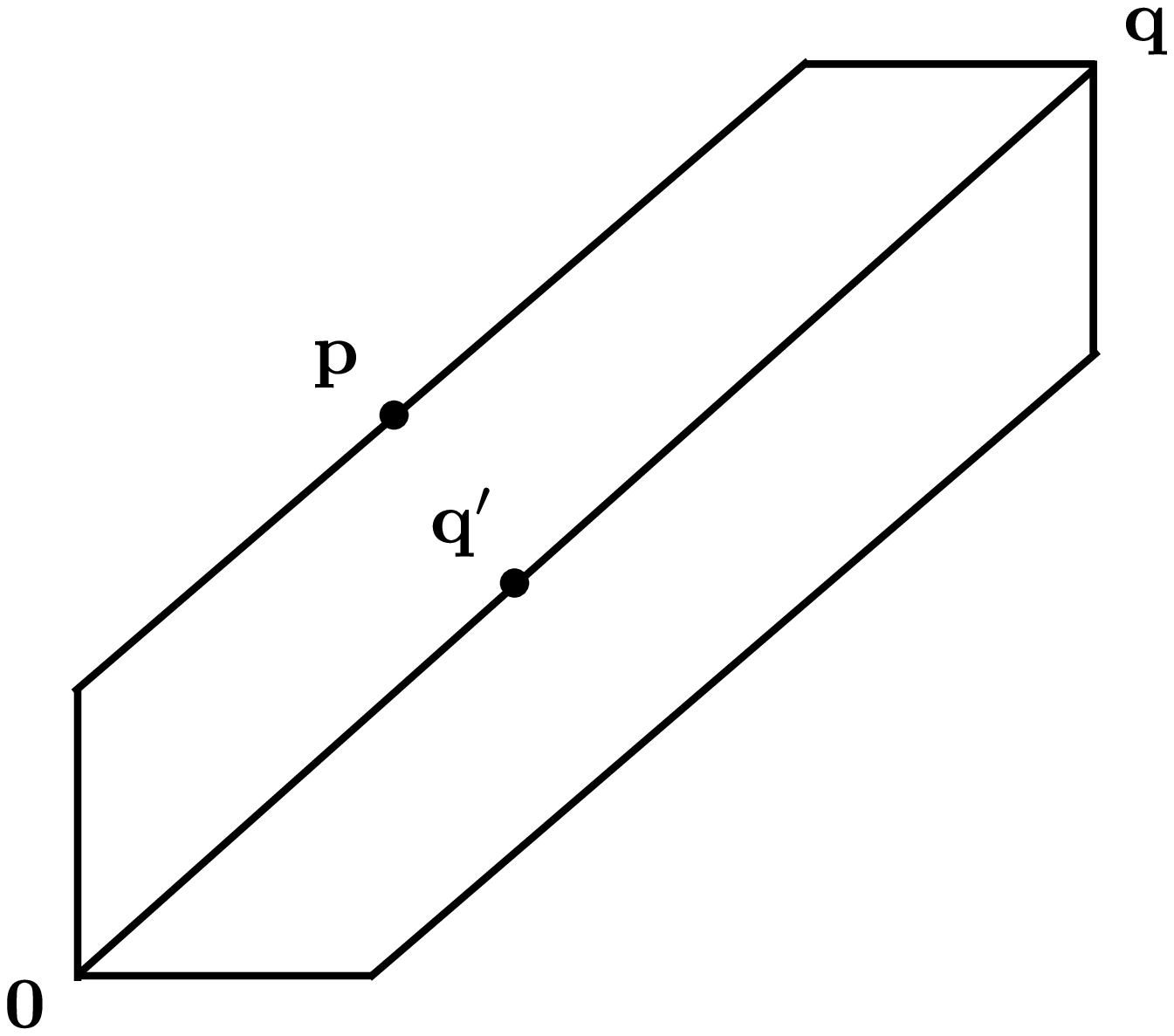}
\caption{Illustration of $\mr p$ and $\mr q'$ in Case 3}
\label{fig:strip_DLPP_3}
\end{minipage}
\end{figure}

\medskip

Case 2: Suppose that the point $\mr p=(\mr p_1,\mr p_2)$ satisfies $\mr q_2-\mr p_2\le \epsilon_2y|\mr q|^{2/3}$. 
The proof is similar to Case 1.

\medskip

Case 3: Suppose that the point $\mr p=(\mr p_1,\mr p_2)$ satisfies $\mr p_1 \ge \epsilon_1 y|\mr q|^{2/3}$ and $\mr q_2-\mr p_2\ge \epsilon_2y|\mr q|^{2/3}$, where $\epsilon_1$, $\epsilon_2$ are determined in the previous two cases. Note that in this case $\mr p$ is of distance at least $\min\{\epsilon_1,\epsilon_2\}y|\mr q|^{2/3}$ to the endpoints of $\Gamma^+$, and hence $\mr p$ and $\mr q-\mr p$ are both in $Q(c''_1,c''_2)$ for some positive constants $c''_1, c''_2$ which are independent of $\mr q$ and $\mr p$.  We will show that
\begin{equation}
\label{eq:aux_012}
d(\mr p)+d(\mr q-\mr p)\le d(\mr q)-c_4y^2|\mr q|^{1/3}
\end{equation}
for some constant $c_4$ independent of $\mr q$. Assuming this inequality,~\eqref{eq:aux_009} follows easily from an argument similar to~\eqref{eq:aux_011}. Therefore it remains to show~\eqref{eq:aux_012}.

To prove~\eqref{eq:aux_012}, it is sufficient to show the following two inequalities: 
\begin{equation}
\label{eq:aux_013}
d(\mr p) \le d(\mr q') -\frac12c_4y^2|\mr q|^{1/3},\qquad d(\mr q-\mr p)\le d(\mr q-\mr q') -\frac12c_4y^2|\mr q|^{1/3},
\end{equation}
where $\mr q'$ is a lattice point neighboring to $\Gamma$ which is 
given by
\begin{equation}
\label{eq:aux_014}
\mr q'_1=\mr p_1+\frac{y|\mr q|^{5/3}}{\sqrt{\mr q_2}(\sqrt{\mr q_1}+\sqrt{\mr q_2})}+O(1),\qquad \mr q'_2=\mr p_2-\frac{y|\mr q|^{5/3}}{\sqrt{\mr q_1}(\sqrt{\mr q_1}+\sqrt{\mr q_2})}+O(1),
\end{equation}
with $O(1)$ terms bounded by $2$. 
Noting that $\mr p_2=\mr p_1 \frac{\mr q_2}{\mr q_1} +\frac{y|\mr q|^{5/3}}{\mr q_1}$, this guarantees the existence of $\mr q'$ neighboring to $\Gamma$ satisfying the above equation. See Figure~\ref{fig:strip_DLPP_3}.
Consider the first inequality of~\eqref{eq:aux_013}. A tedious calculation shows that
\begin{equation}
\label{eq:aux_015}
\begin{split}
&d(\mr q') -d(\mr p) +O(1)\\
& =\left(\sqrt{\mr p_1+\frac{y|\mr q|^{5/3}}{\sqrt{\mr q_2}(\sqrt{\mr q_1}+\sqrt{\mr q_2})}}+\sqrt{\mr p_2-\frac{y|\mr q|^{5/3}}{\sqrt{\mr q_1}(\sqrt{\mr q_1}+\sqrt{\mr q_2})}}\right)^2 -\left(\sqrt{\mr p_1}+\sqrt{\mr p_2}\right)^2\\
&=\left(\mr p_1+\frac{y|\mr q|^{5/3}}{\sqrt{\mr q_2}(\sqrt{\mr q_1}+\sqrt{\mr q_2})}\right)\left(1+\sqrt{\frac{\mr q_2}{\mr q_1}}\right)^2-\left(\sqrt{\mr p_1}+\sqrt{\mr p_1\cdot \frac{\mr q_2}{\mr q_1} +\frac{y|\mr q|^{5/3}}{\mr q_1}}\right)^2\\
&=\frac{y^2|\mr q|^{10/3}}{\mr p_1\mr q_1\sqrt{\mr q_1\mr q_2}}\left(\sqrt{\frac{\mr q_2}{\mr q_1}}+\sqrt{\frac{\mr q_2}{\mr q_1}+\frac{y|\mr q|^{5/3}}{\mr p_1\mr q_1}}\right)^{-2}
\end{split}
\end{equation}
where the term $O(1)$ comes from the $O(1)$ perturbations in~\eqref{eq:aux_014}. 
Since $\mr p_1\ge \epsilon_1 y|\mr q|^{2/3}$,  the right hand side of~\eqref{eq:aux_015} is at least 
\begin{equation}
\frac12 c_4 y^2|\mr q|^{1/3}
\end{equation}
where $c_4$ only depends on $c_1,c_2$ and $\epsilon_1$. Thus we proved the first inequality of~\eqref{eq:aux_013}. The second one follows by observing that if we do the following change of variables the second inequality becomes the first one
\begin{equation}
\begin{split}
\mr q=(\mr q_1,\mr q_2) &\to \tilde {\mr q}=(\mr q_2,\mr q_1),\\
\mr p=(\mr p_1,\mr p_2) &\to \tilde {\mr p}=(\mr q_2-\mr p_2,\mr q_1-\mr p_1),\\
\mr q'=(\mr q'_1,\mr q'_2) & \to \tilde {\mr q}'=(\mr q_2- \mr q'_2,\mr q_1-\mr q'_1).
\end{split}
\end{equation}
\end{proof}

\subsection{Transversal fluctuations of periodic DLPP}
\label{sec:tail_periodic_DLPP}

We now consider  the  periodic DLPP. 
Recall that the period $\mr v=(\mr v_1,\mr v_2)$ is the lattice such that $\mr v_2<0<\mr v_1$ and $w(\mr q+\mr v)=w(\mr q)$ for all lattice points $\mr q$. 
We denote by $\prob_{\mr v}$ the probability measure of the periodic DLPP 
and by $\prob$ the probability measure of the usual DLPP. 
We use the same notations $\mr Q(c_1,c_2)$, $\pi^{max}_{\mr p}(\mr q)$, and $B_{\overline{\mr p\mr q}}(y)$ as in Subsection~\ref{sec:tail_DLPP}. 
Note that  if $\dist( \mr v,\overline{\mr p\mr q}) >2y$, then the random variables $w(\mr r)$ with $\mr r\in B_{\overline{\mr p\mr q}}(y)$ are independent since for any $\mr r,\mr r'\in B_{\overline{\mr p\mr q}}(y)$, we have $\mr r-\mr r'\ne\mr  v$.

\begin{prop}[Estimate for transversal fluctuations in periodic DLPP]
\label{prop:tail_transversal2}
For fixed positive constants $c_1< c_2$ and  $\epsilon$,  there exist positive constants $C$ and $c$ such that
\begin{equation}
\label{eq:tail_transversal2}
\prob_{\mr v}\left(\pi^{max}(\mr q)\subseteq B_{\overline{\mr 0\mr q}}(y|\mr q|^{2/3})\right)\ge 1- e^{-cy^2}
\end{equation}
for all $\mr q\in \mr Q(c_1,c_2)$ satisfying $|\mr q|\ge C$, $y$ satisfying $(\log|\mr q|)^{1/2+\epsilon}\le y \le |\mr q|^{1/3}$, and $\mr v$ satisfying $\dist (\mr v,\overline{\mr 0\mr q} )> \frac32y|\mr q|^{2/3}$. 
\end{prop}

\begin{proof}[Proof of Proposition~\ref{prop:tail_transversal2}]

\begin{figure}
\centering
\begin{minipage}{.4\textwidth}
\includegraphics[scale=0.3]{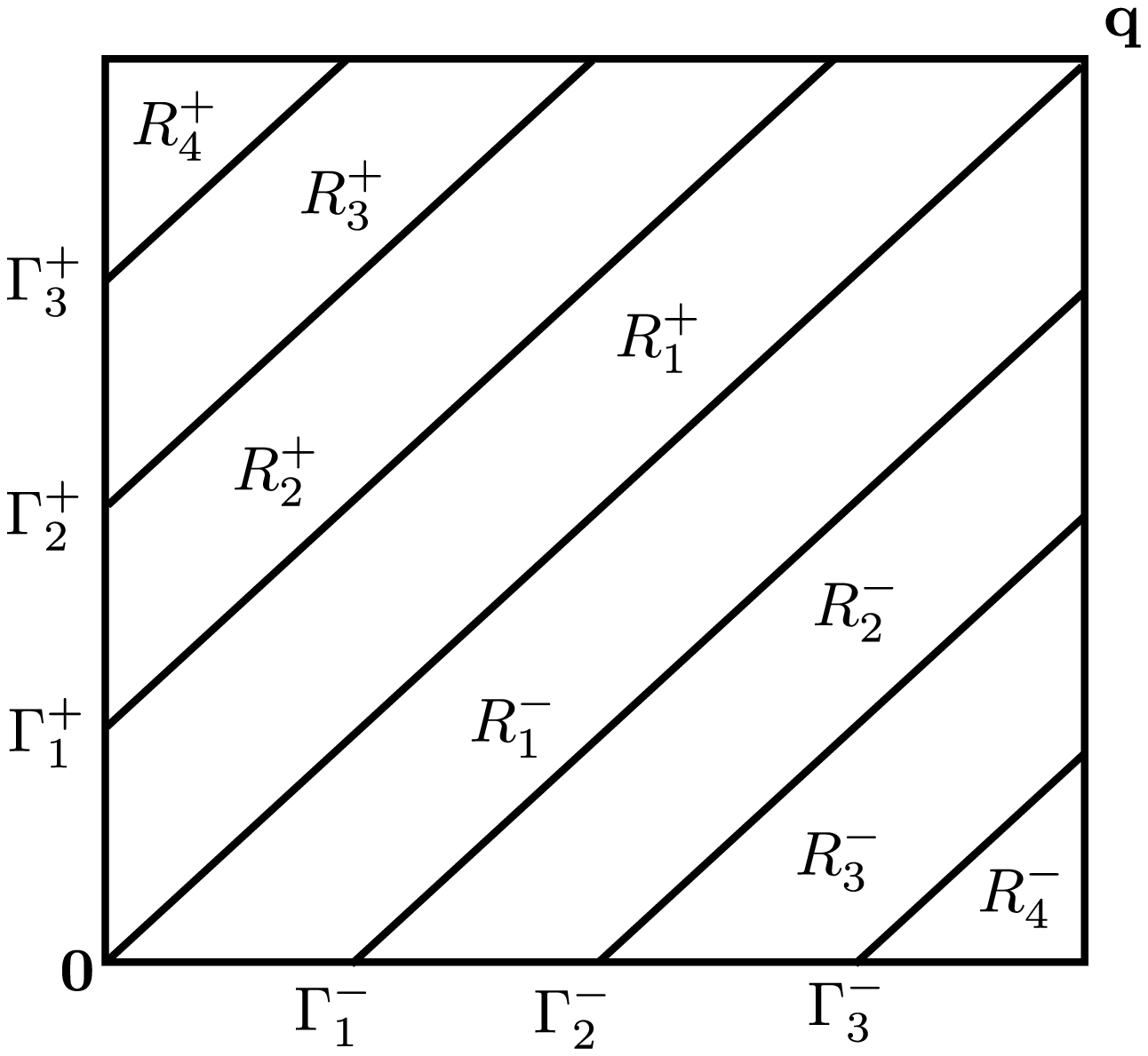}
\caption{Illustration of $\Gamma_j^\pm$ and $R_j^\pm$}
\label{fig:strip_period_DLPP_1}
\end{minipage}
\begin{minipage}{.4\textwidth}
\includegraphics[scale=0.3]{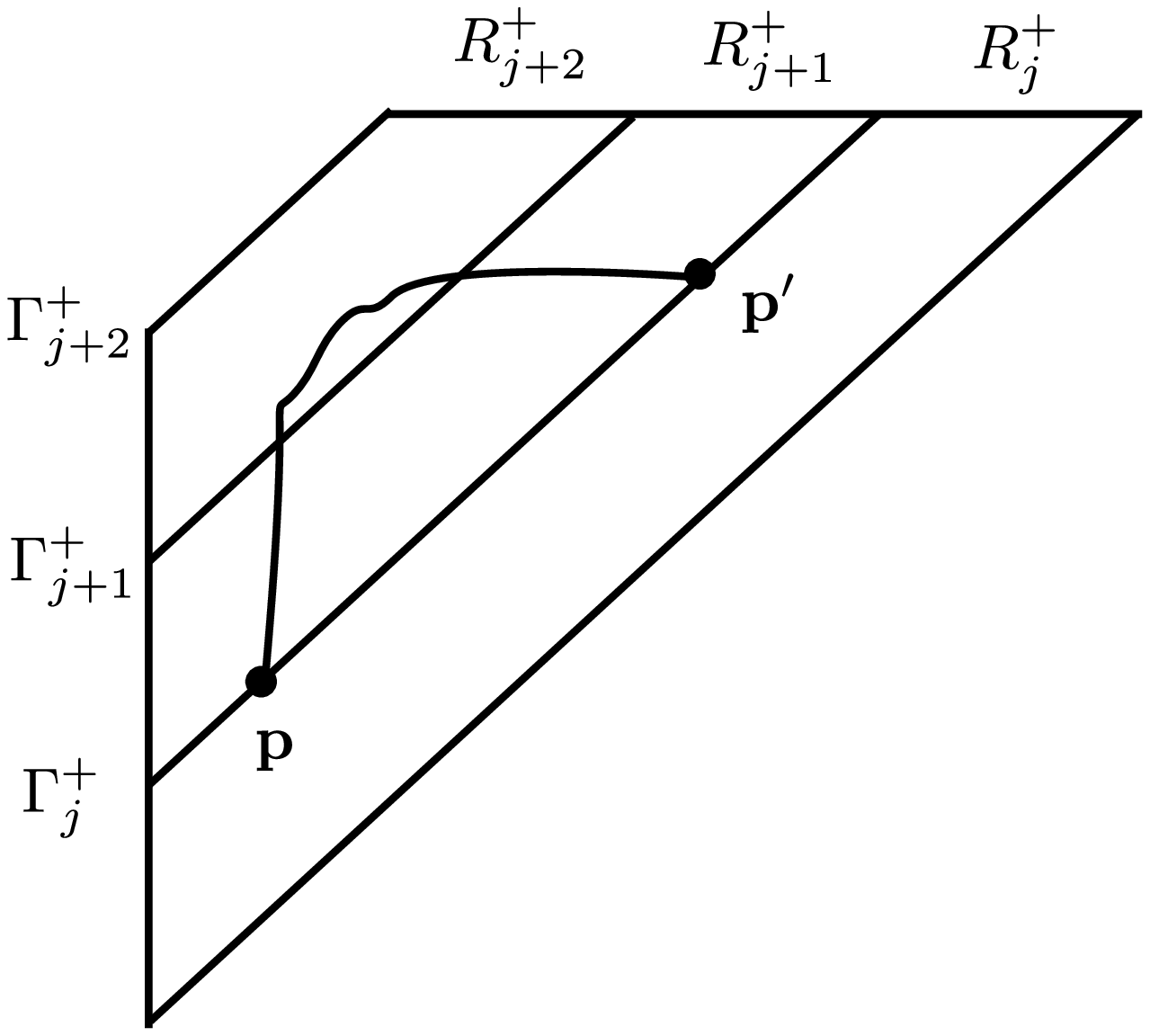}
\caption{Illustration of $\pi^{max}(\mr q)$ and the points $\mr p, \mr p'$}
\label{fig:strip_period_DLPP_2}
\end{minipage}
\end{figure}

Denote by $\Gamma$ the line segment $\overline{\mr 0\mr q}$. 
Let $R$ be the rectangle of which $\Gamma$ is the diagonal. 
We denote the two triangles in $R$ split by $\Gamma$ by $R^+$ and $R^-$ from left to right. 
Set
\begin{equation}
\Gamma_j^{\pm}:=\left\{ u\in R^\pm \, ;  \, \dist ( u,\Gamma )= \frac12jy|\mr q|^{2/3}\right\},\qquad j=1,2,\cdots. 
\end{equation}
We also set $R_j^{+}$ to be the region in $R^+$ bounded by $\Gamma_j^{+}$ and $\Gamma_{j-1}^{+}$ for each $j=1,2,\cdots$. 
We define $R_j^-$ similarly. 
Here $\Gamma_0^{\pm}=\Gamma$. See Figure~\ref{fig:strip_period_DLPP_1} for an illustration.

Note that $B_{\overline{\mr 0\mr q}}(y|\mr q|^{2/3})\cap R=R_2^+\cup R_1^+\cup R_1^-\cup R_2^-$. If $\pi^{max}(\mr q)\not\subseteq B_{\overline{\mr 0\mr q}}(y|\mr q|^{2/3})$, then it intersects  either $\Gamma_2^+$ or $\Gamma_2^{-}$. 
When $\pi^{max}(\mr q)$ intersects $\Gamma_2^+$, there exists the largest $j\ge 1$ and two leftmost lattice points $\mr p$ and $\mr p'$ which are neighboring to $\Gamma_j^\pm$, such that $\pi^{max}(\mr q)$ (1) passes $\mr p$ and $\mr p'$, (2) intersects $\Gamma_{j+1}^+$, (3) does not intersect $\Gamma_{j+2}^+$, and (4) the part of $\pi^{max}(\mr q)$ between $\mr p$ and $\mr p'$ stays in $R_{j+1}\cup R_{j+2}$. See Figure~\ref{fig:strip_period_DLPP_2} for an illustration. This implies that the maximal path from $\mr p$ to $\mr p'$ in the region $R_{j}^+\cup R_{j+1}^+\cup R_{j+2}^+$ intersects $\Gamma_{j+1}^+$. 
Note that from the condition $\frac32y|\mr q|^{2/3}<\dist(\mr v,\overline{\mr 0 \mr q})$, 
 the random variables $w(\mr p)$ are independent for $\mr p$ in 
three consecutive regions $R_j^{\pm}$. 
In particular, $w(\mr p)$ are independent for $\mr p$ in the region $R_{j}^+\cup R_{j+1}^+\cup R_{j+2}^+$, 
and hence the probability that the maximal path from $\mr p$ to $\mr p'$ in the region $R_{j}^+\cup R_{j+1}^+\cup R_{j+2}^+$ intersects $\Gamma_{j+1}^+$ is at most the probability that the maximal path from $\mr p$ to $\mr p'$ in usual DLPP intersects $\Gamma_{j+1}^+$. 
Hence, denoting by $A_j^\pm$ the set of lattice points in $R$ which are neighboring to $\Gamma_j^\pm$, 
we find from  Proposition~\ref{prop:tail_transversal} that 
\begin{equation}
\begin{split}
&\prob_{\mr v}\left(\pi^{max}(\mr q)\mbox{ intersects }\Gamma_2^+\right)\\
\le &\sum_{j\ge 1}\sum_{\mr p,\mr p'\in A_j^+} \prob_{\mr v}\left(\pi^{max}_{\mr p}(\mr p')\mbox{ stays in }R_{j}^+\cup R_{j+1}^+\cup R_{j+2}^+ \mbox{ and intersects }\Gamma_j^+\right)\\
\le &\sum_{j\ge 1}\sum_{\mr p,\mr p'\in A_j^+} \prob\left(\pi^{max}_{\mr p}(\mr p') \mbox{ intersects }\Gamma_j^+\right)\\
\le &\sum_{j\ge 1}\sum_{\mr p,\mr p'\in A_j^+} e^{-cy^2|\mr q|^{4/3}|\mr p-\mr p'|^{-4/3}} \le \sum_{j\ge 1}\sum_{\mr p,\mr p'\in A_j^+} e^{-cy^2}
\end{split}
\end{equation}
where 
the constant $c$ is independent of $\mr p,\mr p'$ and $\mr q$. 
Since there are at most $O(|\mr q|^{1/3}\cdot |\mr q|^2)$ terms in the sum and $y\ge \log(|\mr q|)^{1/2+\epsilon}$, we obtain 
\begin{equation}
\prob_{\mr v}\left(\pi^{max}(\mr q)\mbox{ intersects }\Gamma_2^+\right) \le e^{-cy^2}
\end{equation}
for a different constant $c$. Similarly the probability that $\pi^{max}(\mr q)$ intersects $\Gamma_2^-$ is bounded by $e^{-cy^2}$.
Hence, we obtain~\eqref{eq:tail_transversal2}.
\end{proof}

\subsection{Comparison between DLPP and periodic DLPP}

We now compare the last passage time in DLPP and periodic DLPP. 
We first embed both models in the same probability space so that the last passage times in two models can be compared directly.

Let $\mr v$ be a lattice point. Suppose $\mr R$ is a set of lattice points such that
\begin{equation}
\label{eq:aux_2016_08_21_01}
\left(\mr R+\mr v\right)\cap\mr R=\emptyset, \qquad  \intZ^2=\cup_{i\in\intZ} (\mr R+i\mr v).
\end{equation}
Consider the probability space in which every lattice point $\mr p$ is assigned with an \iid exponential random variables $w(\mr p)$. 
We define new variables $\tilde w(\mr r+i\mr v)= w(\mr r)$ for all $\mr r\in\mr R$ and $i\in\intZ$.
The assumptions on $\mr R$ imply that that $\tilde w(\mr p)$ is well-defined, is defined for all $\mr p\in\intZ$, and  satisfies the periodicity $\tilde w(\mr p)=\tilde w(\mr p+\mr v)$ for all $\mr p\in\intZ^2$. 
Let  $G_{\mr p}(\mr q)$ and $H^{(\mr R)}_{\mr p}(\mr q)$ be the last passage times from $\mr p$ to $\mr q$ with respect to the weights $w$ and $\tilde w$, respectively. 
Here we put an index $\mr R$ in $H$ in order to indicate the dependence on the choice of $\mr R$. 
As before, we write $H^{(\mr R)}(\mr q)=H^{(\mr R)}_{\mr 0}(\mr q)$ and $G(\mr q)=G_{\mr 0}(\mr q)$.

We have the following result. 

\begin{prop}
	\label{prop:comparison_DLPP_periodic_DLPP}
	Let $c_1< c_2$ and $\lambda$ be all fixed positive constants. 
	Then there exist positive constants $C$ and $c$ such that
	\begin{equation}
	\label{eq:aux_2016_08_20_01}
	\prob\left(\bigcap_{|\mr q'-\mr q|\le \lambda|\mr q|^{2/3}}\left\{H^{(\mr R)}(\mr q')= G(\mr q')\right\}\right)\ge 1- e^{-c|\mr v|^2|\mr q|^{-4/3}}
	\end{equation}
	for all $\mr q\in\mr Q(c_1,c_2)$ such that $|\mr q|\ge C$, and for all $\mr R=\intZ^2\cap \left\{x\mr v+y\mr q ; -1/2<x\le 1/2, y\in\realR\right\}$
	where 
	$\mr v=(\mr v_1,\mr v_2)$ is any point in $\intZ^2$ satisfying $\mr v_2<0<\mr v_1$ and $|\mr v|\ge |\mr q|^{2/3}\left(\log |\mr q|\right)^{1/2+\epsilon}$.
\end{prop}

\begin{proof}
	Using Propositions~\ref{prop:tail_transversal} and~\ref{prop:tail_transversal2}, we have\begin{equation}
	\prob\left(H^{(\mr R)}(\mr q')\ne  G(\mr q')\right)\le e^{-c|\mr v|^2|\mr q|^{-4/3}}
	\end{equation}
	for all $\mr q, \mr q'\in\intZ^2$ satisfying $|\mr q|\ge C$ and $|\mr q'-\mr q|\le \lambda|\mr q|^{2/3}$. Here the constants $C$ and $c$ are independent of $\mr q$ and $\mr q'$. Since there are only $O(|\mr q|^{4/3})$ such lattice points $\mr q'$, and $|\mr v|^2|\mr q|^{-4/3}\ge (\log|\mr q|)^{1+2\epsilon}$, we obtain~\eqref{eq:aux_2016_08_20_01} (with different $C$ and $c$).
\end{proof}

Clearly, for arbitrary $\mr R$ satisfying the conditions~\eqref{eq:aux_2016_08_21_01}, we have
\begin{equation}
\prob_{\mr v}(H_{\mr p}(\mr q)=k)=\prob\left(H^{(\mr R)}_{\mr p}(\mr q)=k\right)
\end{equation}
for all lattice points $\mr p$ and $\mr q$, and all $k\in\intZ_{\ge 1}$. 
As a corollary of Proposition~\ref{prop:comparison_DLPP_periodic_DLPP}, we have 
\begin{equation}
\label{eq:aux_2016_08_21_02}
\prob_{\mr v}\left( \frac{H(\mr q)-d(\mr q)}{s(\mr q)} \le x\right) \to \FGUE(x) 
\end{equation}
for each $x\in \realR$ 
as $\mr q\in \mr Q(c_1,c_2)$ satisfies $|\mr q|\to \infty$ and $\mr v\ge |\mr q|^{2/3}(\log|\mr q|)^{1/2+\epsilon}\to \infty$.

\section{Proof of theorems} 
\label{sec:proof}

\subsection{Map from periodic TASEP  to periodic DLPP}

The standard map from infinite TASEP to the usual DLPP extends easily to a map from periodic TASEP  to periodic DLPP, which we explain now. 
Let $\mr v=(L-N,-N)$. This vector will represent the period of  the periodic DLPP. 
The initial condition of the periodic TASEP gives rise to the boundary path of the periodic DLPP as follows. 
Given the initial condition $x_k(0)$ of the periodic TASEP, let $\Lambda$ be the lattice path in $\intZ^2$  defined by the set of points $u=(u_1,u_2)\in\intZ^2$ satisfying either 
\begin{equation}
i+1+x_{N-i}(0)\le u_1\le i+x_{N+1-i}(0), \qquad u_2=i+1
\end{equation}
or
\begin{equation}
u_1=i+x_{N+1-i}(0), \quad i\le u_2\le i+1
\end{equation}
for some $i\in\intZ$. 
Then $\Lambda$ is a lattice path 
whose lower-left corners are $\left(i+x_{N+1-i}(0), i\right),\ i\in\intZ$.
It is invariant under the translation by $\mr v$, i.e., $\Lambda+\mr v=\Lambda$. 
Especially, for the periodic step initial condition $\Lambda$ is a staircase shape lattice path with lower left corners $\mr c_i:=\left(1,1\right)+i\mr v$, and for the flat initial condition $\Lambda$ is a ``flat''  lattice path which consists of consecutive vertical and horizontal line segments with length $1$ and $\rho^{-1}-1$ respectively. See Figures~\ref{fig:periodic_TASEP_boundaries_1} and~\ref{fig:periodic_TASEP_boundaries_2} for an illustration.

\begin{figure}
\centering
\begin{minipage}{.4\textwidth}
\includegraphics[scale=0.2]{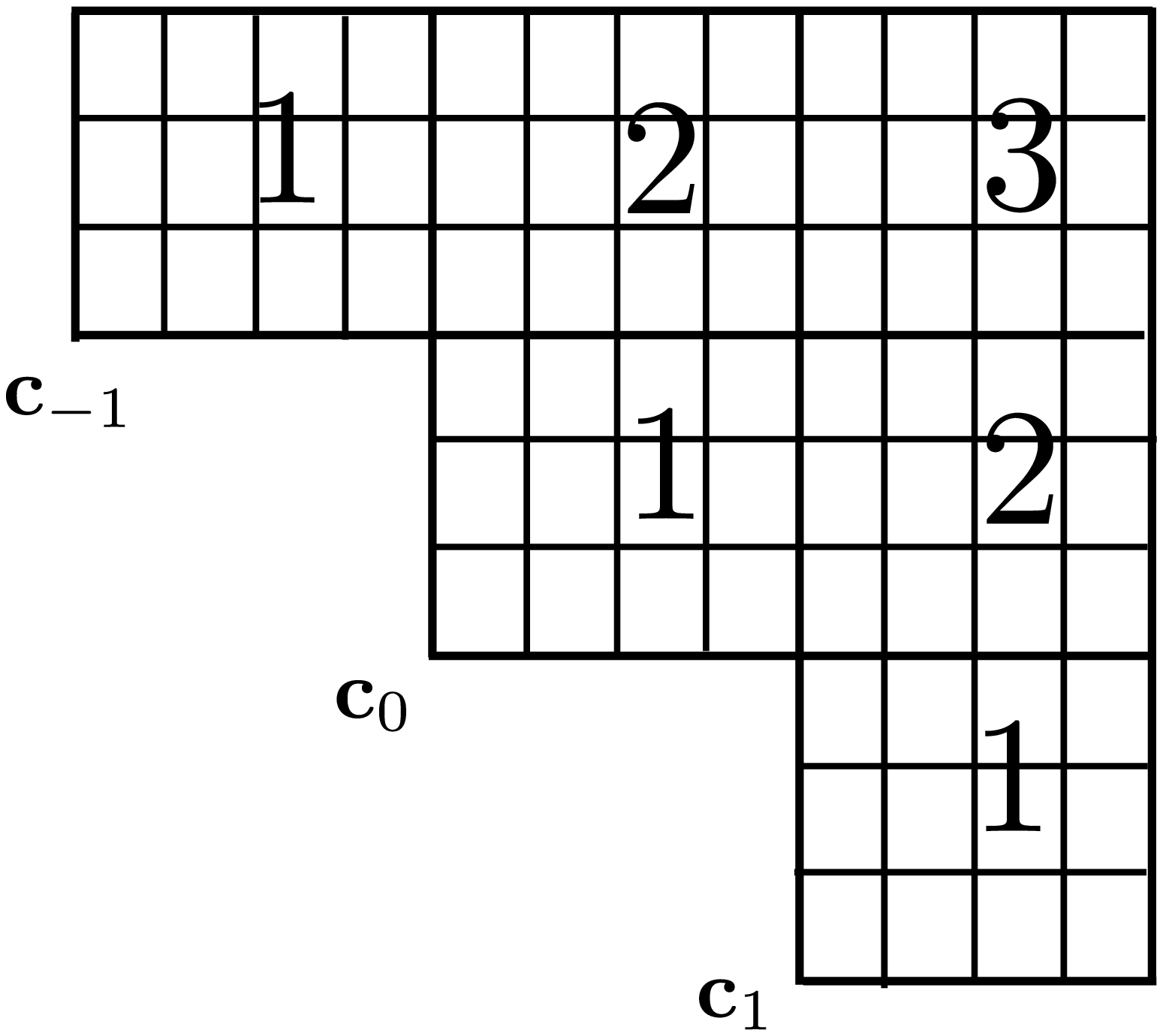}
\caption{The periodic DLPP corresponding to the periodic TASEP 
	with step initial condition. Here $L=7$ and $N=3$. The blocks with the same numbers are identical.}
\label{fig:periodic_TASEP_boundaries_1}
\end{minipage}
\quad
\begin{minipage}{.4\textwidth}
\includegraphics[scale=0.2]{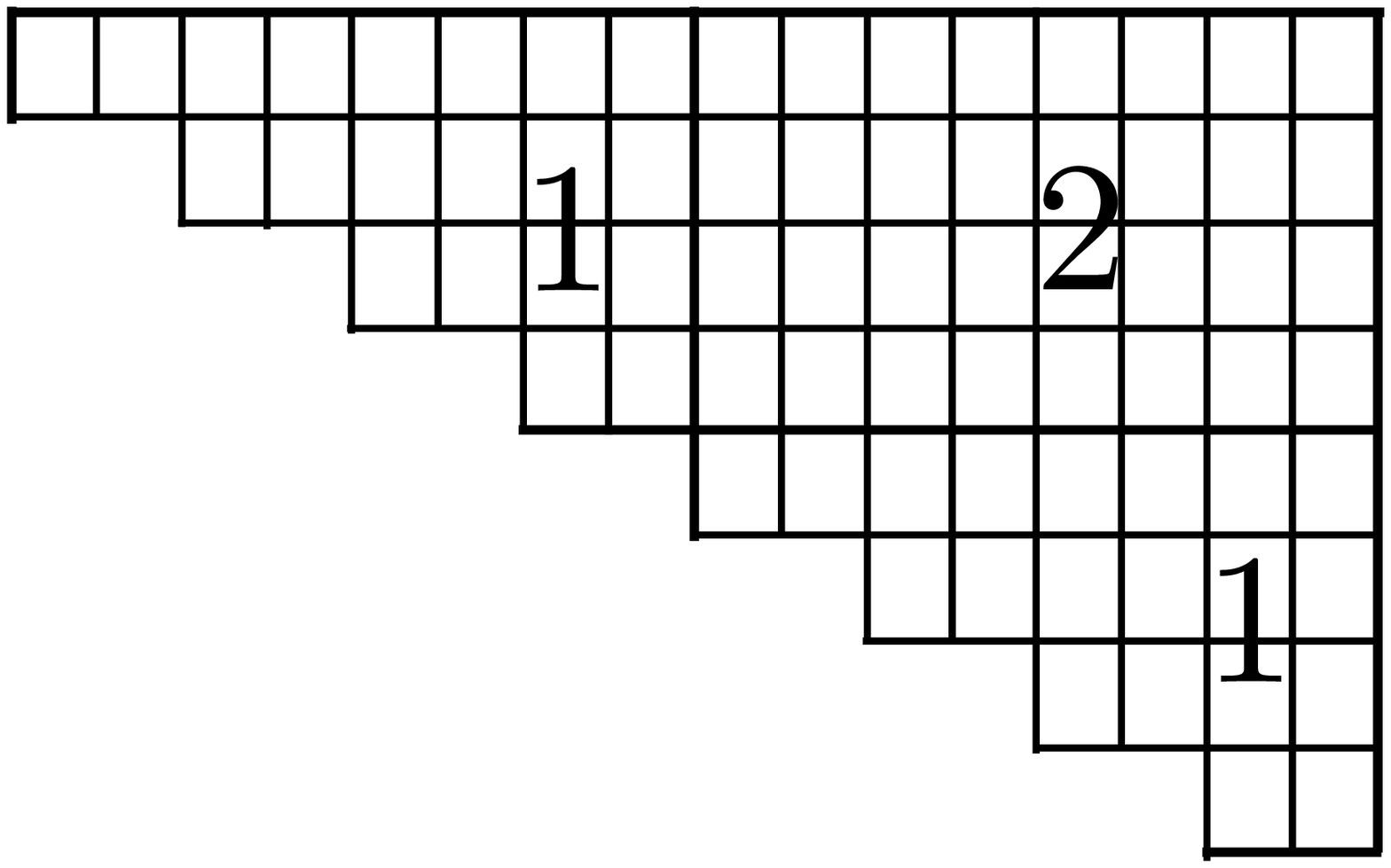}
\caption{The periodic DLPP corresponding to the periodic TASEP 
	 with flat initial condition. Here $L=12$ and $N=4$. The blocks with the same numbers are identical.}
\label{fig:periodic_TASEP_boundaries_2}
\end{minipage}
\end{figure}

We define the random variables $w(i,j)$ associated to site $(i,j)$ by the waiting time of the $(N+1-j)$-th particle, after the right neighboring site becomes empty, stays at the site $i-j$. 
Then $w$ satisfies $w(\mr p)=w(\mr p+\mr v)$ for all $\mr p$ on the upper right side of the boundary path $\Lambda$. See Figures~\ref{fig:periodic_TASEP_boundaries_1} and~\ref{fig:periodic_TASEP_boundaries_2}. 
We set $w(\mr p)=-\infty$ for $\mr p$ on the lower left side of $\Lambda$.


If we remove the restrictions $w(\mr q)=w(\mr v+\mr q)$ in the above setting and suppose $w(\mr q)$ are \iid for all $\mr q$ on $\Lambda$ or at the upper right side of $\Lambda$, then we obtain the usual DLPP with the boundary $\Lambda$. Here we assume $\Lambda$ is the same lattice path defined by the initial condition of the periodic TASEP. Then this new DLPP corresponds to the infinite TASEP with the same initial condition. Therefore in terms of DLPP, the difference between periodic TASEP and infinite TASEP 
is the periodicity  $w(\mr p)=w(\mr p+\mr v)$.
\vspace{0.4cm}

The relation between $x_j(t)$ in the infinite TASEP and the line-to-point last passage time of the usual DLPP is well-known. 
Define the line-to-point last passage time in the usual DLPP model as
\begin{equation}
\label{eq:aux_2016_08_17_01}
G_{\Lambda}(\mr q) =\max_{\mr p\in\Lambda} G_{\mr p}(\mr q).
\end{equation}
Then, for integer $a$ satisfying $a\ge x_k(0)$, 
\begin{equation}
\label{eq:aux_2016_08_18_02}
x_k(t)\ge a \mbox { in the infinite TASEP}\Longleftrightarrow G_\Lambda(\mr q)\le t \mbox{ in the corresponding DLPP},
\end{equation}
where
\beq \label{eq:qdefhr}
\mr q=(N+a-k,N+1-k).
\eeq
For the periodic model, the relation is  same. 
We set 
\begin{equation}
H_{\Lambda}(\mr q)=\max_{\mr p\in\Lambda} H_{\mr p}(\mr q)
\end{equation}
where $H_{\mr p}(\mr q)$ is defined in~\eqref{eq:aux_017}.
Then, for integer $a$ satisfying $a\ge x_k(0)$, 
\begin{equation}
\label{eq:aux_018}
x_k(t)\ge a \mbox{ in the periodic TASEP}\Longleftrightarrow H_\Lambda(\mr q)\le t \mbox{ in the corresponding periodic DLPP},
\end{equation}
where $\mr q$ is defined in~\eqref{eq:qdefhr}.

We use~\eqref{eq:aux_2016_08_18_02} and~\eqref{eq:aux_018} to prove our main theorems in the next three subsections. The proofs of some technical lemmas are postponed to Section~\ref{sec:others}.
We will show that the sub-relaxation time scale implies that Proposition~\ref{prop:tail_transversal2} is applicable. 
Hence under the sub-relaxation time scale, the last passage time in the periodic DLPP has the same 
distribution as the last passage time in the usual DLPP with high probability. See Proposition~\ref{prop:comparison_DLPP_periodic_DLPP}. 
Hence in the leading order, we have formally
\begin{equation}
H_{\mr p}(\mr q) \approx G_{\mr p}(\mr q)\approx d(\mr q-\mr p).
\end{equation} 
For the periodic step initial condition, the geometry of the boundary path $\Lambda$ implies that 
\begin{equation}
	d(\mr q-\mr p) =  \max_{i} d(\mr q- \mr c_i), 
\end{equation}
where $\mr c_i:=\left(1,1\right)+i\mr v$ are the lower-left corners of $\Lambda$. 
It is a simple calculation to check which $i$ gives the largest contribution using the explicit formula of the function $d$. 
We find that for each $i$, 
there is a curve given by the set of points $\mr q$ such that $d(\mr q- \mr c_i)= d(\mr q- \mr c_{i+1})$.
If $\mr q$ is away from these curves,  there is unique maximizer $i$. 
For $\mr q$ on a curve, there are two maximizers. 
These are illustrated in Figure~\ref{fig:periodic_lpp4}. 
In terms of periodic TASEP and infinite TASEP, the curves correspond to the the space-time trajectory of the shocks. 
The maximizing indices are $i=j_N$  for Theorem~\ref{thm:limiting_process_step} (away from shock) and  and $i=j_N, j_N-1$ for Theorem~\ref{thm:limiting_process_step_shock} (near the shock). 
Then formally
\begin{equation}
H_{\Lambda}(\mr q)\approx G_{\Lambda}(\mr q)\approx d(\mr q-c_{j_N})
\end{equation}
for the leading order.
We show that $H_{\Lambda}(\mr q)\approx G_{\Lambda}(\mr q)$ even for the fluctuation term. 
The fluctuation term is different  for $\mr q$ away from or near the shock curves. 
 The flat initial condition case is simpler. 

\begin{figure}
\centering
\includegraphics[scale=0.5]{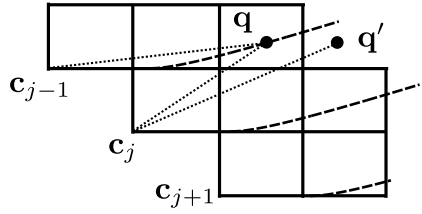}
\caption{
Illustration of the maximizers: $\mr q$ is on a dashed curve and there are two maximizers $d(\mr q-\mr c_{j-1})\approx d(\mr q-\mr c_{j})$.  $\mr q'$ is away from the dashed curves, and there is a unique maximizer $d(\mr q'-\mr c_j)$.}
\label{fig:periodic_lpp4}
\end{figure}

\subsection{Proof of Theorem~\ref{thm:limiting_process_flat}}

For the flat initial condition~\eqref{eq:flat_ic}, 
the  boundary  path  $\Lambda$ has lower-left corners 
\beq \tilde{\mr c}_i:=((1-\rho^{-1})i+\rho^{-1}(N+1),\, i), \qquad i\in\intZ.
\eeq
By assumption, $t=t_N$ is a sequence satisfying $t_N\le CN^{3/2-\epsilon}$ and $\lim_{N\to\infty}t_N=\infty$. To prove~\eqref{eq:limiting_process_flat}, we need to show that for any fixed $k\in\intZ_{\ge 1}$ and $u_1,\cdots,u_k\in\realR$, and $x_1,\cdots,x_k\in\realR$, 
\begin{equation}
\label{eq:aux_2016_08_18_01}
\begin{split}
	&\lim_{N\to\infty}\prob_{\mr v}\left(\bigcap_{j=1}^k\left\{x_{[\kappa_1 u_jt_N^{2/3}]}(t_N)\ge \left[(1-\rho)t_N+\rho^{-1}\kappa_1u_j t_N^{2/3}-\sigma_1x_jt_N^{2/3}\right]\right\}\right)
	=\prob\left(\bigcap_{j=1}^k\{\mathcal{A}_1(u_j)\le x_j\}\right)
\end{split}
\end{equation}
in the periodic TASEP model. Here we use the same notation $\prob_{\mr v}$ as in periodic DLPP to denote the probability measure for the periodic TASEP. For notational convenience, we assume that $\kappa_1u_jt_N^{2/3}$ and $(1-\rho)t_N+\rho^{-1}\kappa_1u_jt_N^{2/3}-\sigma_1x_jt_N^{2/3}$ for $j=1,\cdots,k$ are all integers; this assumption does not affect our proof since $O(1)$ perturbations in these terms do not change our argument below. 
Noting~\eqref{eq:aux_018} and~\eqref{eq:qdefhr}, we define the following subset of $\intZ^2$: 
\begin{equation}
\mr S=\left\{\left( N+ (1-\rho)t_N+(\rho^{-1}-1)\kappa_1u_jt_N^{2/3}-\sigma_1x_jt_N^{2/3}, \, N+1-\kappa_1u_jt_N^{2/3}\right) \in \intZ^2 ; \, j=1,\cdots,k\right\} .
\end{equation}
From~\eqref{eq:aux_018}, Theorem~\ref{thm:limiting_process_flat} follows if we show that 
\begin{equation}
\label{eq:aux_047}
\lim_{N\to\infty}\prob_{\mr v}\left(\bigcap_{\mr q\in\mr S}\{H_\Lambda(\mr q)\le t_N\}\right)=\prob\left(\bigcap_{j=1}^k\{\mathcal{A}_1(u_j)\le x_j\}\right)
\end{equation}

On the other hand, the infinite TASEP with the same flat initial condition satisfies the result~\eqref{eq:aux_2016_08_18_01}  (with the subscript $\mr v$ removed)
\cite{Sasamoto05, Borodin-Ferrari-Prahofer-Sasamoto07} (see the footnote~\ref{ft:flat_TASEP} under the discussions of Theorem~\ref{thm:limiting_process_flat}). 
In terms of DLPP, it means that 
\begin{equation}
\label{eq:aux_2016_08_18_03}
\lim_{N\to\infty}\prob\left(\bigcap_{\mr q\in\mr S}\{G_\Lambda(\mr q)\le t_N\}\right)=
\prob\left(\bigcap_{j=1}^k\{\mathcal{A}_1(u_j)\le x_j\}\right)
\end{equation}
which is a DLPP analog of~\eqref{eq:aux_047}. 
We prove Theorem~\ref{thm:limiting_process_flat} by showing that the left hand sides of~\eqref{eq:aux_047} and~\eqref{eq:aux_2016_08_18_03} are equal. 

Define an index set
\beq
	I:=\{i\in\intZ \, ; \, -[N/4]<N-i-\rho^2 t_N \le [N/4]\}.
\eeq
The next  lemma shows that the main contribution on the left hand sides of~\eqref{eq:aux_047} and~\eqref{eq:aux_2016_08_18_03} comes from the indices in $I$. The proof of this lemma is given in Section~\ref{sec:proof_Lemma_flat_1}.

\begin{lm}
\label{lm:flat_trivial}
For any $j\in\intZ$ which is not in $I$, and any $\mr q\in\mr S$, we have
\begin{equation}
\label{eq:aux_051}
\prob \left(G_{\tilde{\mr c}_j}(\mr q) >\max_{i\in I }G_{\tilde{\mr c}_i }(\mr q)\right) <e^{-t_N^{c\epsilon}},\qquad \prob_{\mr v}\left(H_{\tilde{\mr c}_j}(\mr q) >\max_{i\in I }H_{\tilde{\mr c}_i }(\mr q)\right) <e^{-t_N^{c\epsilon}}
\end{equation}
for large enough $N$, where   $c$ is a constant independent of $j$,  $t_N$ and $N$, and $\epsilon$ is the constant defined in Theorem~\ref{thm:limiting_process_flat}. 
\end{lm}

By Propositions~\ref{prop:tail_transversal} and~\ref{prop:tail_transversal2}, the maximal paths from the lattice points $\{\tilde{\mr c}_i; i\in I\}$ to $\mr q$ in both periodic DLPP and usual DLPP are concentrated with high probability in a strip whose vertical length is $N/2+O(t_N^{2/3+\epsilon''})$, where $\epsilon''>0$ is a constant such that $t_N^{2/3+\epsilon''}\ll N/2$. In this strip all the entries are \iid for both models. Therefore we have
\begin{equation}
\prob_{\mr v}\left(\bigcap_{\mr q\in\mr S}\{\max_{i\in I}H_{\tilde{\mr c}_i }(\mr q)\le t_N\}\right)-\prob\left(\bigcap_{\mr q\in\mr S}\{\max_{i\in I}G_{\tilde{\mr c}_i }(\mr q)\le t_N\}\right)\to 0
\end{equation} 
as $N\to\infty$. Together with Lemma~\ref{lm:flat_trivial}, we proved that the left hand sides of~\eqref{eq:aux_047} and~\eqref{eq:aux_2016_08_18_03} are equal.


\subsection{Proof of Theorem~\ref{thm:limiting_process_step}}
\label{sec:proof_theorem_step}

For the periodic step initial condition~\eqref{eq:step_ic}, the boundary path  $\Lambda$ has lower-left corners 
\beq
\label{eq:aux_2016_08_15_03}
\mr c_i=(1,1)+i\mr v= (1+i(L-N),1-iN), \qquad i\in\intZ.
\eeq
For the infinite TASEP, we need to show that for any fixed $k\in\intZ_{\ge1}$ and $u_1,\cdots,u_k\in\realR$, and $x_1,\cdots,x_k\in\realR$, 
\begin{multline}
\label{eq:aux_2016_08_09_01}
	\lim_{N\to\infty}\prob\left(\bigcap_{i=1}^k\left\{x_{[\alpha N+\kappa_2 u_it_N^{2/3}]}(t_N)\ge \left[(1-2\mu)t_N+\rho^{-1}j_NN+\mu^{-1}\kappa_2u_it_N^{2/3}-\sigma_2x_it_N^{1/3}\right]\right\}\right)\\
	=\prob\left(\bigcap_{j=1}^k\left\{\mathcal{A}_2(u_j)-u_j^2\le x_j\right\}\right).
\end{multline}
For notational convenience, we again  assume 
that $\alpha N+\kappa_2u_it_N^{2/3}$ and $(1-2\mu)t_N+\rho^{-1}j_NN+\mu^{-1}\kappa_2u_it_N^{2/3}-\sigma_2x_it_N^{1/3}$ are integers for all $i=1,\cdots,k$. 
Let $\mr S$ denote the set of lattice points 
\begin{equation}
\label{eq:aux_2016_08_15_01}
	\left((1-2\mu)t_N+(\rho^{-1}j_N+1-\alpha)N-(1-\mu^{-1})\kappa_2u_it_N^{2/3}-\sigma_2x_it_N^{1/3}, 
	\, (1-\alpha)N-\kappa_2 u_it_N^{2/3}+1\right) 
\end{equation}
where $ i=1,\cdots,k$.
The result~\eqref{eq:aux_2016_08_09_01} follows if we show
\begin{equation}
\label{eq:aux_2016_08_18_05}
\lim_{N\to\infty}\prob\left(\bigcap_{\mr q\in\mr S}\{G_\Lambda(\mr q)\le t_N\}\right) =
\prob\left(\bigcap_{j=1}^k\left\{\mathcal{A}_2(u_j)-u_j^2\le x_j\right\}\right). 
\end{equation}
Similarly, for the periodic TASEP,  we need to show that 
\begin{equation}
\label{eq:aux_020}
\lim_{N\to\infty}\prob_{\mr v}\left(\bigcap_{\mr q\in\mr S}\{H_\Lambda(\mr q)\le t_N\}\right) =
\prob\left(\bigcap_{j=1}^k\left\{\mathcal{A}_2(u_j)-u_j^2\le x_j\right\}\right).
\end{equation}

From the geometry of the boundary path $\Lambda$, the maximal paths from $\Lambda$ to $\mr q$ are paths from some corners to $\mr q$. 
The following lemma 
shows that 
if the particle is away from a shock (which is the assumption of Theorem~\ref{thm:limiting_process_step}), this corner is $\mr c_{j_N}$ with high probability. 
Recall that $j_N$ is the parameter introduced in the statement of the theorem which measures the number of encounters with shocks by the particle $[\alpha N]$. 
The lemma is proved in Section~\ref{sec:proof_Lemma_1}.

\begin{lm}
\label{lm:corners_trivial}
For all $i\in \intZ$ satisfying $i\ne j_N$, and all $\mr q\in\mr S$, we have
\begin{equation}
\label{eq:aux_024}
\prob\left(G_{\mr c_i}(\mr q) >G_{\mr c_{j_N}}(\mr q)\right) <e^{-t_N^{c\epsilon}},\qquad \prob_{\mr v}\left(H_{\mr c_i}(\mr q) >H_{\mr c_{j_N}}(\mr q)\right) <e^{-t_N^{c\epsilon}}
\end{equation}
for large enough $N$, where $c$ 
is a constant only depending on $u_1,\cdots,u_k, x_1,\cdots,x_k$, and $\epsilon$ is the constant defined in Theorem~\ref{thm:limiting_process_step}.
\end{lm}

Now we prove~\eqref{eq:aux_2016_08_18_05} and~\eqref{eq:aux_020}.
Using Lemma~\ref{lm:corners_trivial}, 
\begin{equation}
\label{eq:aux_021}
\begin{split}
\prob\left(\bigcap_{\mr q\in\mr S}\{G_{\Lambda}(\mr q)\le t_N\}\right)
&= 1- \prob\left(\bigcup_{\mr q\in\mr S}\{G_{\Lambda}(\mr q)> t_N\}\right)\\
&\ge 1-\prob\left(\bigcup_{\mr q\in\mr S}\{G_{\mr c_{j_N}}(\mr q)> t_N\}\right)-\sum_{i\ne j_N}\sum_{\mr q\in\mr S}\prob\left(G_{\mr c_i}(\mr q)> G_{\mr c_{j_N}}(\mr q)\right)\\
&=\prob\left(\bigcap_{\mr q\in\mr S}\{G_{\mr c_j}(\mr q)\le  t_N\}\right)-\sum_{i\ne j_N}\sum_{\mr q\in\mr S}\prob\left(G_{\mr c_i}(\mr q)> G_{\mr c_{j_N}}(\mr q)\right)\\
&\ge \prob\left(\bigcap_{\mr q\in\mr S}\{G_{\mr c_j}(\mr q)\le  t_N\}\right)-\sum_{\mr q\in\mr S}c'N^{-1}|\mr q|e^{-t_N^{c\epsilon}}
\end{split}
\end{equation}
since there are at most $O(N^{-1}|\mr q|)$ corners $\mr c_i$  such that $\prob\left(G_{\mr c_i}(\mr q)> G_{\mr c_j}(\mr q)\right)\ne 0$. Here $c'$ is a positive constant. 
We also have the trivial bound 
\begin{equation}
\label{eq:aux_022}
\prob\left(\bigcap_{\mr q\in\mr S}\{G_{\Lambda}(\mr q)\le t_N\}\right) \le \prob\left(\bigcap_{\mr q\in\mr S}\{G_{\mr c_{j_N}}(\mr q)\le  t_N\}\right).
\end{equation}
Thus in order to prove~\eqref{eq:aux_2016_08_18_05}, it is sufficient to show that 
\begin{equation}
\label{eq:aux_2016_08_18_06}
\lim_{N\to\infty}\prob\left(\bigcap_{\mr q\in\mr S}\{G_{\mr c_{j_N}}(\mr q) \le  t_N\}\right)=
\prob\left(\bigcap_{j=1}^k\left\{\mathcal{A}_2(u_j)-u_j^2\le x_j\right\}\right).
\end{equation}
But $\mr q$ in $\mr S$ is of form~\eqref{eq:aux_2016_08_15_01} with some $i=1,\cdots,k$. 
Hence 
\begin{equation}
\mr q-\mr c_{j_N}=\left((1-\mu)^2t_N-(1-\mu^{-1})\kappa_2u_2t_N^{2/3}-\sigma_2x_it_N^{1/3}+o(t_N^{1/3}), \mu^2t_N-\kappa_2u_it_N^{2/3}\right)
\end{equation}
for some $i=1,\cdots,k$, where the $o(t_N^{1/3})$ term  equals to $j_N([\rho^{-1}N]-L)\ll t_N^{1/3}$.
This is exactly the same framework for the Airy$_2$ process limit of multi-point distribution in the DLPP (see \cite{Prahofer-Spohn02, Johansson03, Borodin-Ferrari08}), and it is well-known that 
\begin{equation}
\lim_{N\to\infty}\prob\left(\bigcap_{\mr q\in\mr S}\{G(\mr q-\mr c_{j_N}) \le  t_N\}\right)= 
\prob\left(\bigcap_{j=1}^k\left\{\mathcal{A}_2(u_j)-u_j^2\le x_j\right\}\right).
\end{equation}
Hence~\eqref{eq:aux_2016_08_18_06} is proved.

For the periodic TASEP,~\eqref{eq:aux_020} follows if we show that 
\begin{equation}
\label{eq:aux_2016_08_18_07}
\lim_{N\to\infty}\prob_{\mr v}\left(\bigcap_{\mr q\in\mr S}\{H_{\mr c_{j_N}}(\mr q) \le  t_N\}\right)= 
\prob\left(\bigcap_{j=1}^k\left\{\mathcal{A}_2(u_j)-u_j^2\le x_j\right\}\right).
\end{equation}
Now, Proposition~\ref{prop:comparison_DLPP_periodic_DLPP} shows that the left hand sides of both~\eqref{eq:aux_2016_08_18_07} and~\eqref{eq:aux_2016_08_18_06} are equal. Thus we obtain~\eqref{eq:aux_2016_08_18_07}.

\subsection{Proof of Theorem~\ref{thm:limiting_process_step_shock}}
\label{sec:proof_theorem_step_2}

\subsubsection{Parts (a) and (b)}\label{sec:theorem3ab}
Consider part (a). 
The proof of Theorem~\ref{thm:limiting_process_step} applies without any change if we have Lemma~\ref{lm:corners_trivial} with $t_N=s_{j_N}N$ and the under the restriction that 
$u_1,\cdots,u_k >0$. 
The proof of Lemma~\ref{lm:corners_trivial} still applies in this set-up only after a small change; see the discussions in Subsection~\ref{sec:proof_Lemma_1_others}.
Part (b) is similar. 

\subsubsection{Part (c)}

 

The argument is similar to the proof of Theorem~\ref{thm:limiting_process_step}. 
In this case,  we only consider the maximal path to a single point, and hence we do not need the set  $\mr S$ in the proof. 
Instead we only define one lattice point
\begin{equation}
\label{eq:aux_2016_08_23_01}
	\mr q=\left((1-2\mu)s_{j_N}+(\rho^{-1}j_N+1-\alpha)N-\sigma_2xs_{j_N}^{1/3},(1-\alpha)N+1\right).
\end{equation}
After that, the proof proceeds same as before with the following lemmas in place of Lemma~\ref{lm:corners_trivial} and the equations~\eqref{eq:aux_2016_08_18_06} and~\eqref{eq:aux_2016_08_18_07}. 
Their proofs are given in Section~\ref{sec:inftaseprel} and we obtain the part (c).


\begin{lm}
	\label{lm:corners_trivial2}
	For all $i\ge j_N+1$, we have
	\begin{equation}
	\label{eq:aux_044}
	\prob\left(G_{\mr c_i}(\mr q) >G_{\mr c_{j_N}}(\mr q)\right) <e^{-t_N^{c\epsilon}},\qquad 
	\prob_{\mr v}\left(H_{\mr c_i}(\mr q) >H_{\mr c_{j_N}}(\mr q)\right) <e^{-t_N^{c\epsilon}}
	\end{equation}
	for large enough $N$, where $c$ is a constant independent of $j_N$, $t_N$ and $N$, and $\epsilon$ is the constant defined in Theorem~\ref{thm:limiting_process_step}. For all $i \le j_N-2$, we have~\eqref{eq:aux_044} with $\mr c_{j_N}$ replaced by $\mr c_{j_{N}-1}$.
\end{lm}

\begin{lm}
\label{lm:corner_main2}  We have 
\begin{equation}
\label{eq:aux_045}
\begin{split}
\lim_{N\to\infty}\prob\left(\max\{G_{\mr c_{j_N}}(\mr q), G_{\mr c_{j_N-1}}(\mr q)\}\le t_N\right) 
&= \FGUE(x)\FGUE(r^{-1}x),\\
\lim_{N\to\infty}\prob_{\mr v}\left(\max\{H_{\mr c_{j_N}}(\mr q), H_{\mr c_{j_N-1}}(\mr q)\}\le t_N\right) 
&= \FGUE(x)\FGUE(r^{-1}x).
\end{split}
\end{equation}
\end{lm}

\begin{rmk}
Ferrari and Nejjar \cite{Ferrari-Nejjar15} obtained a simple general theorem which shows that the fluctuations at the shock are given by the maximum of two independent random variables under certain assumptions. 
The difficult part is to check the assumptions, and they did it for a few examples. 
The infinite TASEP with periodic step initial condition has two features which are not present in those examples: (i) there are growing number of boundary corners $\mr c_i$'s while there were only two corners in the examples of \cite{Ferrari-Nejjar15}, 
and (ii) we are interested in the case when the end point of the maximal path is of order up to $\ell^{3/2-\epsilon}$ if $\ell$ denotes the distances between the consecutive boundary corners while in \cite{Ferrari-Nejjar15} the end point is of order $O(\ell)$. 
It might be possible to check the assumptions of the general theorem of Ferrari and Nejjar \cite{Ferrari-Nejjar15} for our case, but we instead proceed more directly using some of the ideas in \cite{Ferrari-Nejjar15} instead of trying to check their assumptions. 
Furthermore, we give an uniform proof for both infinite TASEP and the periodic TASEP, the later of which is not discussed in \cite{Ferrari-Nejjar15}. 
\end{rmk}

\subsection{One point fluctuations of infinite TASEP with periodic step initial condition in relaxation time scale}\label{sec:inftaseprel}

In this subsection, we discuss the infinite TASEP with periodic step initial condition at the relaxation and super-relaxation time scales mentioned in Subsection~\ref{sec:reldis}.
We first consider relaxation time scale.  
Suppose $N=N_n$ and $L=L_n$ are two sequences of integers such that
\begin{equation}
\rho_n:=\frac{N_n}{L_n}=\rho+O(L_n^{-1})
\end{equation}
as $n\to\infty$.
 Recall that the periodic step initial condition is
\begin{equation}
x_{i+lN_n}(0)=-N_n+i+lL_n
\end{equation}
for all $1\le i\le N_n$ and all $l\in\intZ$. 
Let $\gamma\in \realR$ and $\tau>0$ be two fixed constants.
Set
\begin{equation}
\label{eq:aux_2016_08_19_01}
t_n=\frac{L_n}{\rho_n}\left[\frac{\tau \sqrt{\rho_n}}{\sqrt{1-\rho_n}} L_n^{1/2}\right]+\frac{L_n}{\rho_n}\gamma +\frac{L_n}{\rho_n}\left(1-\frac{k_n}{N_n}\right)
\end{equation}
where $k_n$ is an arbitrary integer sequence such that $1\le k_n\le N_n$. 
Note that $t_n=O(L_n^{3/2})$ and hence this is the relaxation time scale. 
In this case, we expect that 
\begin{equation}
\label{eq:aux_2016_08_19_02}
\begin{split}
	&\lim_{n\to\infty}\prob\left(\frac{\left(x_{k_n}(t_n)-x_{k_n}(0)\right)-(1-\rho_n)t_n+(1-\rho_n)L_n(1-k_n/L_n)}{\rho_n^{-1/3}(1-\rho_n)^{2/3}t_n^{1/3}}\ge -x\right) \\
	&\qquad =\prob\left(\bigcap_{2\tau^{2/3}u-\gamma\in\intZ}\left\{\mathcal{A}_2(u)-u^2\le x \right\}\right),
\end{split}
\end{equation}
as we explain now.

This follows from the following corresponding result for the usual DLPP: 
\begin{equation}
\label{eq:aux_2016_08_19_03}
	\lim_{n\to\infty}\prob\left( \max_{i} G_{\mr c_i}(\mr q)\le t_n\right)=\prob\left(\bigcap_{2\tau^{2/3}u-\gamma\in\intZ}\left\{\mathcal{A}_2(u)-u^2\le x \right\}\right),
\end{equation}
where 
\begin{equation}
\mr c_i=(1,1)+iv= (1+i(\rho_n^{-1}-1)N_n,1-iN_n), \qquad i\in\intZ,
\end{equation}
and the point $\mr q=(\mr q_1,\mr q_2)$ satisfies
\begin{equation}
\mr q_1=(1-\rho_n)t_n - (1-\rho_n)L_n(1-k_n/L_n)-\rho_n^{-1/3}(1-\rho_n)^{2/3}xt_n^{1/3},
\end{equation}
and $\mr q_2=N_n+1-k_n$. 
Note that 
\begin{equation}
\begin{split}
\mr q-\mr c_i
			&=\Bigg((1-\rho_n)^2t_n+(1-\rho_n)L_n\left(\left[\frac{\tau \sqrt{\rho_n}}{\sqrt{1-\rho_n}} L_n^{1/2}\right]-i+\gamma\right)-\rho_n^{-1/3}(1-\rho_n)^{2/3}xt_n^{1/3}-1,\\
			&\quad\rho_n^2 t_n-\rho_nL_n\left(\left[\frac{\tau \sqrt{\rho_n}}{\sqrt{1-\rho_n}} L_n^{1/2}\right]-i+\gamma\right)\Bigg).
\end{split}
\end{equation}
Since $L_n=\rho_n^{1/3}(1-\rho_n)^{1/3}\tau^{-2/3}t_n^{2/3}+O(t_n^{1/3})$, 
we expect from the Airy$_2$ convergence of the usual DLPP \cite{Prahofer-Spohn02, Johansson03, Borodin-Ferrari08} that formally
\begin{equation}
\frac{G_{\mr c_i}(\mr q)-t_n}{\rho_n^{-1/3}(1-\rho_n)^{-1/3} t_n^{1/3}}\approx \mathcal{A}_2\left(u\right)-u^2 -x
\end{equation}
where 
\begin{equation}
u=\frac1{2\tau^{2/3}}\left(\left[\frac{\tau \sqrt{\rho_n}}{\sqrt{1-\rho_n}} L_n^{1/2}\right]-i+\gamma\right)
\end{equation}
for all $i\in \intZ$. 
Thus we formally obtain~\eqref{eq:aux_2016_08_19_03}. 
To make this rigorous, we need to prove the convergence in all $i$. 
We do not pursue this direction here. 


The scalings in~\eqref{eq:aux_2016_08_19_01} and~\eqref{eq:aux_2016_08_19_02} are same as those for the periodic TASEP with the same initial condition; see equations (3.13) and (3.14) in \cite{Baik-Liu16}.
The limiting distribution, $F(x)=\prob\left(\bigcap_{2\tau^{2/3}u-\gamma\in\intZ}\left\{\mathcal{A}_2(u)-u^2\le x \right\}\right)$, however, is  presumably different from the one for the periodic TASEP obtained in \cite{Prolhac16, Baik-Liu16}.
We note that for the discrete-time infinite TASEP with $\rho=1/2$, the fluctuations under very general initial conditions were studied by Corwin, Liu, and Wang \cite{Corwin-Liu-Wang14}.
Assuming that the same holds for continuous-time infinite TASEP with general $\rho$, the relaxation time scale with periodic step initial condition formally fits with the framework of their result: it corresponds to the case of a discrete delta function as the initial profile. This implies, formally,~\eqref{eq:aux_2016_08_19_02}.

The super-relaxation case is, again formally, the case when $\tau = +\infty$ in the above analysis. 
Then the right-hand side of~\eqref{eq:aux_2016_08_19_03} is expected to be 
\beq
	\prob\left( \bigcap_{u\in \realR}\left\{\mathcal{A}_2(u)-u^2\le x \right\}\right)
	= 	\prob\left( \sup_{u\in \realR} (\mathcal{A}_2(u)-u^2) \le x \right).
\end{equation}
It is known that this is same as the GOE Tracy-Widom distribution \cite{Johansson03} after a simple scaling.

\section{Proof of lemmas in Section~\ref{sec:proof}}
\label{sec:others}

We prove the lemmas used in the previous section. 
We first prove Lemmas~\ref{lm:corners_trivial} and~\ref{lm:corners_trivial2} in Subsection~\ref{sec:proof_Lemma_1}, and  then Lemma~\ref{lm:corner_main2} in Subsection~\ref{sec:proof_Lemma_3}. 
Finally we prove Lemma~\ref{lm:flat_trivial} in Subsection~\ref{sec:proof_Lemma_flat_1}.


Throughout this section, we use the notation $c$ to denote a positive constant which is independent of the parameters $N$, $j_N$ and $t_N$. 
Even if the constant is different from one place to another, we may use the same notation $c$ as long as it does not depend on $N$, $j_N$ and $t_N$. 
On the other hand, the notations $c_1$, $c_2$, etc. with a subscription denote absolute constants which do not change from one place to another. 
We also suppress the subscript $N$ in $j_N$ and $t_N$ in this section for notational convenience.

We fix two positive constants $c_1<c_2$ and consider the uniformity of the estimates in $\mr q\in \mr Q(c_1,c_2)$. 
Recall~\eqref{eq:Qdez} for the definition of $\mr Q(c_1,c_2)$. 
We first prove a comparison lemma. 


\begin{lm}
\label{lm:comparison_DLPP}
Fix $0<c_1<c_2$. Let $\epsilon''$ be a fixed positive constant. 
Then there exist positive constants $C$ and $c$ which only depend on $c_1,c_2$ and $\epsilon''$, such that
\begin{equation}
\label{eq:aux_038}
\prob_{\mr v}\left(H(\mr q)\le H(\mr q')\right)\le e^{-c|\mr q|^{\epsilon''}}
\end{equation}
for all $\mr q, \mr q'\in \mr Q(c_1,c_2)$ and $\mr v=(\mr v_1,\mr v_2)\in\intZ^2$ 
satisfying $|\mr q|\ge C$, $|\mr q'|\ge C$,  
\begin{equation}
	d(\mr q)>d(\mr q')+|\mr q|^{1/3+\epsilon''},
\end{equation} 
$\mr v_2<0<\mr v_1$, and $|\mr v|\ge |\mr q|^{2/3+\epsilon''/2}$. 
In particular, we have
\begin{equation}
\label{eq:aux_2016_08_21_03}
\prob \left(G(\mr q)\le G(\mr q')\right)\le e^{-c|\mr q|^{\epsilon''}}.
\end{equation}
\end{lm}
\begin{proof}
We have 
\begin{equation}
\label{eq:aux_039}
\begin{split}
&\prob_{\mr v}\left(H(\mr q)\le H(\mr q')\right)\\
&\le  1-\prob_{\mr v}\left(H(\mr q)>d(\mr q)-\frac{1}{2}|\mr q|^{1/3+\epsilon''}, \, H(\mr q')<d(\mr q')+\frac{1}{2}|\mr q|^{1/3+\epsilon''}\right)\\
&\le  2-\prob_{\mr v}\left(H(\mr q)>d(\mr q)-\frac{1}{2}|\mr q|^{1/3+\epsilon''}\right)-\prob_{\mr v}\left(H(\mr q')<d(\mr q')+\frac{1}{2}|\mr q|^{1/3+\epsilon''}\right) .
\end{split}
\end{equation}
Note that $|\mr v|\ge |\mr q|^{2/3+\epsilon''/2}\gg |\mr q|^{2/3}\log|\mr q|$. Using Proposition~\ref{prop:comparison_DLPP_periodic_DLPP}, the right hand side of~\eqref{eq:aux_039} can be replaced by
\begin{equation}
	2-\prob\left(G(\mr q)>d(\mr q)-\frac{1}{2}|\mr q|^{1/3+\epsilon''}\right)-\prob\left(G(\mr q')<d(\mr q')+\frac{1}{2}|\mr q|^{1/3+\epsilon''}\right)+e^{-c|q|^{\epsilon''}} .
\end{equation}Combining with the tail estimates of DLPP in Lemma~\ref{prop:tail_estimate}, we obtain~\eqref{eq:aux_038}.
\end{proof}

\subsection{Proof of Lemma~\ref{lm:corners_trivial} and~\ref{lm:corners_trivial2}}
\label{sec:proof_Lemma_1}

We present our proof for the usual DLPP only.
We use the transversal estimate,  Proposition~\ref{prop:tail_transversal}, 
and also the inequality~\eqref{eq:aux_2016_08_21_03}. 
The proof applies to the periodic DLPP without any change except that $\prob$ and $G$ are replaced to $\prob_{\mr v}$ and $H$, and we use Proposition~\ref{prop:tail_transversal2} and the inequality~\eqref{eq:aux_038}.

We consider Lemma~\ref{lm:corners_trivial} first. 
Recall that the corners $\mr c_i$ are defined by (see~\eqref{eq:aux_2016_08_15_03})
\begin{equation}
\mr c_i=(1,1)+i\mr v= (i(L-N),-iN) + O(1), \qquad i\in\intZ.
\end{equation}
We need to compare the last passage time $G_{\mr c_i}(\mr q)$ from $\mr c_i$ with arbitrary index $i$ to an arbitrary point $\mr q$ in the set $\mr S$. 
The lattice set $\mr S$ is finite and a point $\mr q=(\mr q_1,\mr q_2)$ in $\mr S$ is of form 
\begin{equation}
\begin{split}
	\mr q_1&=(1-2\mu)t+(\rho^{-1}j+1-\alpha)N-(1-\mu^{-1})\kappa_2 u_\ell t^{2/3}-\sigma_2 x_\ell t^{1/3}
\end{split}
\end{equation}
and
\begin{equation}
\mr q_2=(1-\alpha)N-\kappa_2u_\ell t^{2/3}+1
\end{equation}
for some $\ell=1, \cdots, k$. Since we are in the sub-relaxation time scale, $t_N\le O(N^{3/2-\epsilon})$, 
we see that 
\begin{equation}\label{eq:mrqforonefq}
\begin{split}
	\mr q&= \left( (1-2\mu)t+(\rho^{-1}j+1-\alpha)N, \, (1-\alpha)N \right) + o(N). 
\end{split}
\end{equation}
Thus the leading term $d(\mr q-\mr c_i)$ of $G_{\mr c_i}(\mr q)$ does not depend on $\ell=1, \cdots, k$. 
In order to prove Lemma~\ref{lm:corners_trivial},  
\begin{equation}
\label{eq:aux_2016_08_15_02}
	\prob_{\mr v}\left(G_{\mr c_i}(\mr q) >G_{\mr c_{j}}(\mr q)\right) <e^{-t^{c\epsilon}},
\end{equation}
we need to find $i$ at which $d(\mr q-\mr c_i)$ becomes the largest. 
This turns out to be $i=j$. 
For the case of Lemma~\ref{lm:corners_trivial2}, the values at $i=j$ and $i=j-1$ are same to the leading order. 
In the actual proof, one needs to be careful with the error term $o(N)$ in~\eqref{eq:mrqforonefq} in order to make the argument work for all range of the sub-relaxation time scale $t\ll N^{3/2}$. 
We prove the result for $i=j\pm 1$ first. 
After that, we obtain the result for the case  $|i-j|\ge 2$ from the  case $|i-j|=1$.

In the proof, we  assume that $\mr c_i$ are on the lower-left side of q. 
Otherwise 
the inequality is trivial.

\subsubsection{Proving~\eqref{eq:aux_2016_08_15_02} when $|i-j|=1$ }

We use Proposition~\ref{prop:tail_estimate} and Lemma~\ref{lm:comparison_DLPP}.
To use them, in our case we need to check that $\mr q-\mr c_i\in\mr Q(c_1,c_2)$. 
When $j=0$ and $i=1$ or $j=1$ and $i=2$, it may not be possible to find fixed 
constants $c_1$ and $c_2$ such that $\mr q-\mr c_i\in \mr Q(c_1,c_2)$. 
We consider these cases separately. 

\vspace{0.3cm}

{\bf{Case $(1)$: $j=0$ and $i=1$.}}

\vspace{0.3cm}

 Since we assume that $\mr c_1$ is on the lower left of $\mr q$, we have $\mr q_1\ge (\rho^{-1}-1)N+O(1)$, i.e.,
\begin{equation}
\label{eq:aux_027}
(\sqrt{t}-\sqrt{(1-\alpha)N})^2-(1-\mu^{-1})\kappa_2ut^{2/3}-x\mu^{-1/3}(1-\mu)^{2/3}t^{1/3}\ge (\rho^{-1}-1)N+O(1).
\end{equation}
Note that it is possible that $\mr q_1$ is close to $(\rho^{-1}-1)N+O(1)$, which implies that $\mr q$ is close to the vertical line with the corner $\mr c_1$. If this happens, 
$\mr q-\mr c_1$ is not necessary in any given cone $\mr Q(c_1,c_2)$.

Pick two positive constants $\epsilon_1$ and $\epsilon_2$ such that
\begin{equation}
\label{eq:aux_028}
\left(\sqrt{\rho^{-1}-1}+\sqrt{1-\alpha}\right)^2 > (\sqrt{\epsilon_1}+\sqrt{2-\alpha})^2+\epsilon_2.
\end{equation}
Such constants exist since $\rho^{-1}\ge 2$. Now set $\mr q'=(\mr q'_1,\mr q'_2)=(\left[\epsilon_1N+(\rho^{-1}-1)N\right], N+1-[\alpha N])$, a lattice point which is on the same horizontal line with $\mr q$. See Figure~\ref{fig:compare_corners_1}. Note that our choice of $\mr q'$ guarantees that $\mr q'$ stays in some cone $\mr Q(c_1,c_2)$ for all $N$.

Recall $d(\mr q)$ defined in~\eqref{eq:aux_026}.
Then using~\eqref{eq:aux_027} and~\eqref{eq:aux_028}, we have, for large enough $N$,
\begin{equation}
\label{eq:aux_025}
\begin{split}
d(\mr q-\mr c_0)&\ge \left(\sqrt{(\rho^{-1}-1)N+O(1)}+\sqrt{(1-\alpha)N+O(1)}\right)^2\\
				&> \left(\sqrt{\epsilon_1N+O(1)}+\sqrt{(2-\alpha)N+O(1)}\right)^2+\epsilon_2N\\
				&=d(\mr q'-\mr c_1)+cN.
\end{split}
\end{equation}
Hence, if $\mr q$ is on the left side of $\mr q'$, we have
\begin{equation}
\label{eq:aux_029}
\prob\left(G_{\mr c_1}(\mr q) > G_{\mr c_0}(\mr q)\right)
\le \prob\left(G_{\mr c_1}(\mr q') > G_{\mr c_0}(\mr q)\right) \le e^{-cN^{2/3}},
\end{equation}
where we applied Lemma~\ref{lm:comparison_DLPP}. Note that $|\mr q|=O(N)=O(t)$ in this case. 
Thus, we obtain~\eqref{eq:aux_2016_08_15_02}.

It remains to show~\eqref{eq:aux_2016_08_15_02} when $\mr q$ is on the right side of $\mr q'$. In this case we have $\mr q-\mr c_0, \mr q-\mr c_1 \in \mr Q(c_1,c_2)$ for some  $c_1$ and $c_2$. Moreover, we can check that
\begin{equation}
\label{eq:aux_031}
\begin{split}
d(\mr q -\mr c_1)<d(\mr q-\mr c_0)-c N
\end{split}
\end{equation}
for sufficiently large $N$. In fact, the above inequality, after dropping the smaller order terms, is equivalent to
\begin{equation}
\left(\sqrt{\left(\sqrt{\ttt_N}-\sqrt{1-\alpha}\right)^2-(\rho^{-1}-1)}+\sqrt{2-\alpha}\right)^2<\ttt_N -c.
\end{equation}
Since $\ttt_N$ is in a compact interval $\mathcal{S}_1^{(\epsilon')}$, it is sufficient to show that
\begin{equation}
\label{eq:aux_030}
\sqrt{\left(\sqrt{\ttt_N}-\sqrt{1-\alpha}\right)^2-(\rho^{-1}-1)}+\sqrt{2-\alpha}<\sqrt{\ttt_N}.
\end{equation}
From~\eqref{eq:aux_028}, we find $\sqrt{\ttt_N}>\sqrt{1-\alpha}+\sqrt{\rho^{-1}-1}>\sqrt{2-\alpha}$. Then~\eqref{eq:aux_030} is equivalent to
\begin{equation}
(\sqrt{\ttt_N}-\sqrt{1-\alpha})^2-(\sqrt{\ttt_N}-\sqrt{2-\alpha})^2<\rho^{-1}-1,
\end{equation}
i.e.,
\begin{equation}
\ttt_N<\frac{\sqrt{2-\alpha}+\sqrt{1-\alpha}}{2\rho}=s_1,
\end{equation}
which is obvious since $\ttt_N$ lies in the interval $\mathcal{S}_1^{(\epsilon')}$.
This implies~\eqref{eq:aux_031}. 
We then obtain~\eqref{eq:aux_2016_08_15_02} by applying Lemma~\ref{lm:comparison_DLPP}.
\begin{figure}
\centering
\includegraphics[scale=0.5]{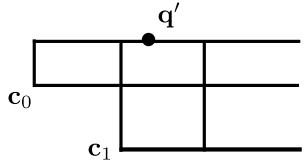}
\caption{Illustration of $\mr q'$ in Case (1) of Step (1)}
\label{fig:compare_corners_1}
\end{figure}

\vspace{0.3cm}
{\bf Case $(2)$: $j=1$ and $i=2$. }

\vspace{0.3cm}

Since we assume $\mr c_2$ is on the lower left of $\mr q$, we have $\mr q_1\ge 2(\rho^{-1}-1)N+O(1)$, i.e.,
\begin{equation}
 (\sqrt{t}-\sqrt{(2-\alpha)N})^2-(1-\mu^{-1})\kappa_2ut^{2/3} -x\mu^{-1/3}(1-\mu)^{2/3}\ge (\rho^{-1}-1)N+O(1).
\end{equation}
Similarly to Case (1), it is possible that $\mr q_1$ is close to $1+2(\rho^{-1}-1)N$, which implies $\mr q$ is close to the vertical line with the corner $\mr c_2$. The proof of~\eqref{eq:aux_2016_08_15_02} in this case is similar to Case (1), and we skip the details. 

\vspace{0.3cm}

{\bf Case $(3)$: $j\ge 2$ and $i=j\pm 1$, or $j=1$ and $i=0$.}

\vspace{0.3cm}
In this case, we first show that there exist positive constants $c_1,c_2$ satisfying $c_1<c_2$ such that $\mr q-\mr c_{i}\in\mr Q(c_1,c_2)$ for sufficiently large $N$. In fact, it is sufficient to show that, when  $N$  is sufficiently large,
\begin{equation}
\label{eq:aux_033}
\frac{\mr q_2+(j-1)N-1}{\mr q_1-(j-1)(L-N)-1}>c_1,
\end{equation}
for all $j\ge 1$ and
\begin{equation}
\label{eq:aux_034}
\frac{\mr q_2+(j+1)N-1}{\mr q_1-(j+1)(L-N)-1}<c_2
\end{equation}
for all $j\ge 2$. The first inequality, after dropping smaller order terms, becomes
\begin{equation}
\frac{j-\alpha}{\left(\sqrt{\ttt_N}-\sqrt{j+1-\alpha}\right)^2+\rho^{-1}-1}>c_1
\end{equation}
for all $N$. Using $\sqrt{\ttt_N}<\sqrt{s_{j+1}}\le \sqrt{j+1-\alpha}+(\rho^{-1}-1)\sqrt{j+2-\alpha}$ and $\sqrt{\ttt_N}>\sqrt{s_{j}}>\sqrt{j+1-\alpha}$, the above inequality is further reduced to
\begin{equation}
\frac{j-\alpha}{j+2-\alpha+\rho/(1-\rho)}>\frac{c_1(1-\rho)^2}{\rho^2}
\end{equation}
for all $j\ge 1$. This holds if we choose $c_1$ satisfying $\frac{c_1(1-\rho)^2}{\rho^2} <\min_{j\ge 1}\frac{j-\alpha}{j+2-\alpha+\rho/(1-\rho)}$.

Similarly we can show the second inequality~\eqref{eq:aux_034} holds for all $j\ge 2$ and sufficiently large $N$, 
if we choose $c_2$ satisfying 
\begin{equation}
\frac{c_2(1-\rho^2)}{\rho^2}>\max_{j\ge 2}\frac{j+2-\alpha}{j-\alpha-\rho/(1-\rho)}.
\end{equation}

Now we want to show that
\begin{equation}
\label{eq:aux_037}
d(\mr q-\mr c_i)<d(\mr q-\mr c_j)-c t^{\frac13+\frac29\epsilon}
\end{equation}
 when $N$ is sufficiently large. This inequality, after dropping smaller order terms, is equivalent to
\begin{equation}
\label{eq:aux_035}
\begin{split}
&\left(\sqrt{(\sqrt{\ttt_N}-\sqrt{j+1-\alpha})^2+(j-i)(\rho^{-1}-1)-(1-\mu^{-1})\kappa_2ut^{2/3}N^{-1}}+\sqrt{i+1-\alpha-\kappa_2 ut^{2/3}N^{-1}}\right)^2 \\
&< \left(\sqrt{(\sqrt{\ttt_N}-\sqrt{j+1-\alpha})^2-(1-\mu^{-1})\kappa_2ut^{2/3}N^{-1}}+\sqrt{j+1-\alpha-\kappa_2 ut^{2/3}N^{-1}}\right)^2 -c t^{\frac13+\frac29\epsilon}N^{-1}.
\end{split}
\end{equation}
We first note that the right hand side of~\eqref{eq:aux_035} equals to (by using $\sqrt{j-\alpha+1}=\mu\sqrt{\ttt_N}$)
\begin{equation}
\ttt_N -ct^{\frac13+\frac29\epsilon}N^{-1}+O(t^{1/3}N^{-1}).
\end{equation}
And note that $c  t^{\frac13+\frac29\epsilon}N^{-1}\le O(N^{-\frac12-\frac29\epsilon^2})<\epsilon'$ for sufficiently large $N$, where $\epsilon'$ is the positive constant defined in Theorem~\ref{thm:limiting_process_step} such that $\ttt_N\in\mathcal{S}_j^{(\epsilon')}$. Together with the  facts $\rho\le 1/2$ and $t^{2/3}N^{-1}\le O(N^{-\frac23\epsilon})$, we have
\begin{equation}
\sqrt{\ttt_N-ct^{\frac13+\frac29\epsilon}N^{-1}}>\frac{\sqrt{j-\alpha}+\sqrt{j+1-\alpha}}{2\rho}>\sqrt{i+1-\alpha-\kappa_2ut^{2/3}N^{-1}}. 
\end{equation}
Hence~\eqref{eq:aux_035} is equivalent to
\begin{multline}
\left(\sqrt{\ttt_N-ct^{\frac13+\frac29\epsilon}N^{-1}}-\sqrt{i+1-\alpha-\kappa_2 ut^{2/3}N^{-1}}\right)^2\\
>(\sqrt{\ttt_N}-\sqrt{j+1-\alpha})^2+(j-i)(\rho^{-1}-1)-(1-\mu^{-1})\kappa_2 ut^{2/3}N^{-1},
\end{multline}
and further to 
\begin{multline}
\label{eq:aux_2016_08_09_02}
\frac{2(i-j)}{\sqrt{i+1-\alpha}+\sqrt{j+1-\alpha}}\left(\frac{\sqrt{i+1-\alpha}+\sqrt{j+1-\alpha}}{2}-\sqrt{\ttt_N}\right)\\
>\kappa_2ut^{2/3}N^{-1}\left(\mu^{-1}-\frac{2\sqrt{\ttt_N}}{\sqrt{i+1-\alpha}+\sqrt{i+1-\alpha-\kappa_2ut^{2/3}N^{-1}}}\right) \\
+ct^{\frac13+\frac29\epsilon}N^{-1}\left(1-\frac{2\sqrt{i+1-\alpha-\kappa_2ut^{2/3}N^{-1}}}{\sqrt{\ttt_N}+\sqrt{\ttt_N-ct^{\frac13+\frac29\epsilon}N^{-1}}}\right).
\end{multline}
Now using the assumptions $i=j\pm 1$ and $\ttt_N\in \mathcal{S}_j^{\epsilon'}$, we know the left hand side of~\eqref{eq:aux_2016_08_09_02} is positive and at least $c\epsilon'\ttt_N^{-1}$. On the other hand, recalling $\mu=\sqrt{j+1-\alpha}/\sqrt{\ttt_N}$, it is a direct to check that the first term on the right hand side of~\eqref{eq:aux_2016_08_09_02} is at most
\begin{equation}
\kappa_2|u|t^{2/3}N^{-1}O\left(\ttt_N^{-1}\right)\ll O(\ttt_N^{-1}).
\end{equation}
And the second term on the right hand side of~\eqref{eq:aux_2016_08_09_02} is at most (by noting $t\le O(N^{3/2-\epsilon})$)
\begin{equation}
O(t^{\frac13+\frac29\epsilon}N^{-1})\ll O(Nt^{-1})=O(\ttt_N^{-1}).
\end{equation}
These three estimates implies that~\eqref{eq:aux_2016_08_09_02} holds for sufficiently large $N$.


By using~\eqref{eq:aux_037} and Lemma~\ref{lm:comparison_DLPP}, we obtain~\eqref{eq:aux_2016_08_15_02} for $|i-j|=1$.

\subsubsection{Proving~\eqref{eq:aux_2016_08_15_02} when $|i-j|\ge 2$}

\begin{figure}
\centering
\begin{minipage}{.4\textwidth}
	\includegraphics[scale=0.4]{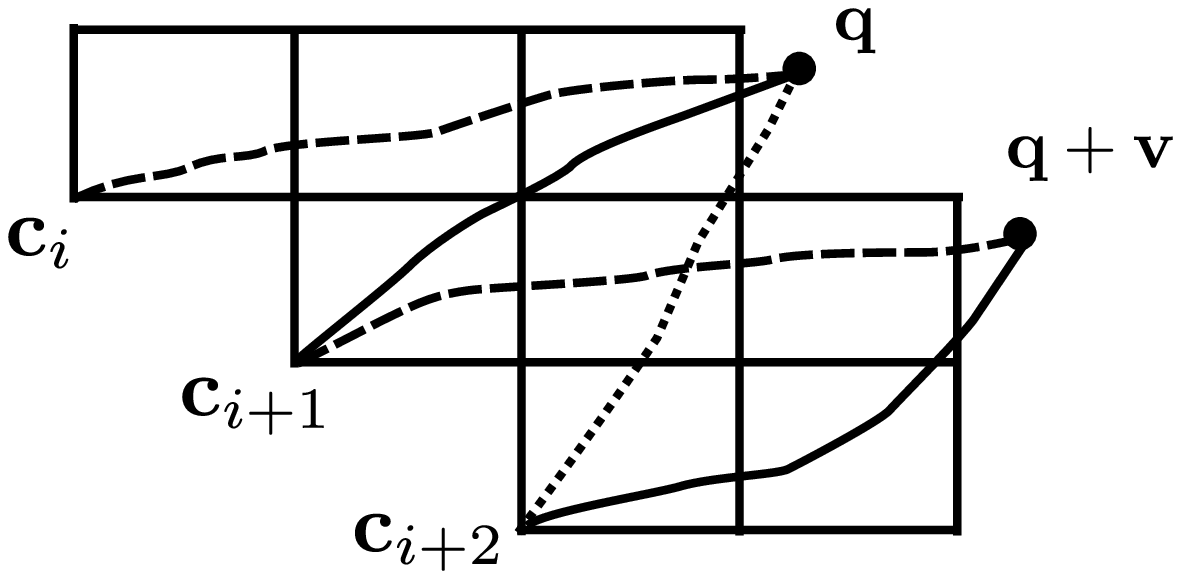}
\end{minipage}
\begin{minipage}{.4\textwidth}
	\includegraphics[scale=0.4]{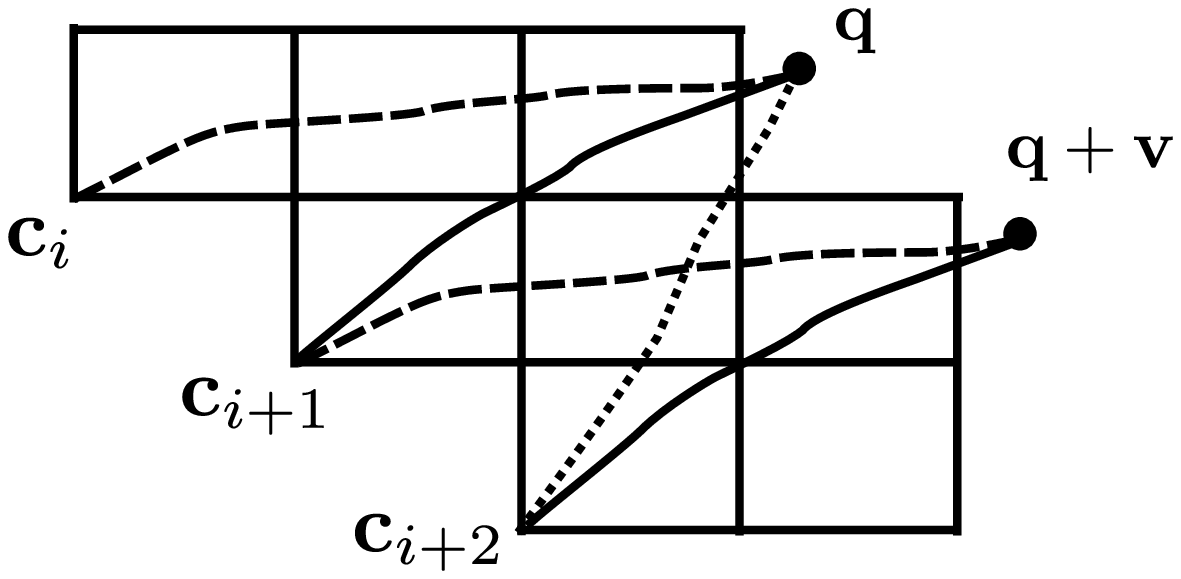}
\end{minipage}
\caption{Illustration of the monotonicity of $\prob(G_{\mr c_i}(\mr q)>G_{\mr c_{i+1}}(\mr q))$ (left) and $\prob_{\mr v}(H_{\mr c_i}(\mr q)>H_{\mr c_{i+1}}(\mr q))$ (right)	}
\label{fig:monotocity}
\end{figure}

We first show the following monotonicity:
\begin{equation}
\label{eq:aux_040}
\prob(G_{\mr c_i}(\mr q)>G_{\mr c_{i+1}}(\mr q))\le \prob(G_{\mr c_{i+1}}(\mr q)>G_{\mr c_{i+2}}(\mr q))
\end{equation}
for all $i$ such that $\mr c_i,\mr c_{i+1}, \mr c_{i+2}$ are all in the lower left of $\mr q$. 
Note that 
\begin{equation}
\label{eq:aux_041}
\prob\left(G_{\mr c_i}(\mr q)>G_{\mr c_{i+1}}(\mr q)\right)=\prob\left(G_{\mr c_{i+1}}(\mr q+\mr v)>G_{\mr c_{i+2}}(\mr q+\mr v)\right)
\end{equation}
due to the translation invariance of DLPP\footnote{In the periodic DLPP case, we even have  $H_{\mr c_i}(\mr q)=H_{\mr c_{i+1}}(\mr q+\mr v)$ and $H_{\mr c_{i+1}}(\mr q)=H_{\mr c_{i+2}}(\mr q+\mr v)$ due to the periodicity. See Figure~\ref{fig:monotocity}.}.
See Figure~\ref{fig:monotocity} for an illustration. 
Since the DLPP model is 2-dimensional, we observe that the maximal path from $\mr c_{i+1}$ to $\mr q+\mr v$ intersects with the maximal path from $\mr c_{i+2}$ to $\mr q$. This implies that 
\begin{equation}
G_{\mr c_{i+1}}(\mr q)+ G_{\mr c_{i+2}}(\mr q+\mr v)\ge G_{\mr c_{i+1}}(\mr q+\mr v)+G_{\mr c_{i+2}}(\mr q),
\end{equation}
and hence 
\begin{equation}
\prob\left(G_{\mr c_{i+1}}(\mr q+\mr v)>G_{\mr c_{i+2}}(\mr q+\mr v)\right)\le \prob(G_{\mr c_{i+1}}(\mr q)>G_{\mr c_{i+2}}(\mr q)).
\end{equation}
Together with~\eqref{eq:aux_041}, this proves~\eqref{eq:aux_040}.

We now prove~\eqref{eq:aux_2016_08_15_02} when $i\le j- 2$. 
The case when $i\ge j+2$ is similar. 
We have
\begin{equation}
\begin{split}
\prob\left(G_{\mr c_i}(\mr q)>G_{\mr c_j}(\mr q)\right)
	&\le \prob\left(G_{\mr c_{i}}(\mr q)>G_{\mr c_{i+1}}(\mr q)\right)+\prob\left(G_{\mr c_{i+1}}(\mr q)>G_{\mr c_j}(\mr q)\right)\\
	&\le \prob\left(G_{\mr c_{j-1}}(\mr q)>G_{\mr c_j}(\mr q)\right)+\prob\left(G_{\mr c_{i+1}}(\mr q)>G_{\mr c_j}(\mr q)\right)
\end{split}
\end{equation}
where we used the monotonicity~\eqref{eq:aux_040} a few times to obtain the second inequality. Using the above inequality recursively we find
\begin{equation}
\prob\left(G_{\mr c_i}(\mr q)>G_{\mr c_j}(\mr q)\right)\le (j-i)\prob\left(G_{\mr c_{j-1}}(\mr q)>G_{\mr c_j}(\mr q)\right).
\end{equation}
Then we apply~\eqref{eq:aux_2016_08_15_02} with $i=j-1$ and note that $j\le  O(t/L)\ll O(t)$. We thus obtain
\begin{equation}
\prob\left(G_{\mr c_i}(\mr q)>G_{\mr c_j}(\mr q)\right)\le e^{-t^{c\epsilon}}.
\end{equation}
Hence the proof of Lemma~\ref{lm:corners_trivial} is complete. 

\subsubsection{Proof of Lemma~\ref{lm:corners_trivial2}} 


Proof for Lemma~\ref{lm:corners_trivial2} is similar. 
In this case the time sequence satisfies $t=s_{j}N$ instead of $t$ being in between $(s_{j}+\epsilon') N$ and $(s_{j+1}-\epsilon')N$. And the point $\mr q$ is given by~\eqref{eq:aux_2016_08_23_01} which is  same as~\eqref{eq:aux_2016_08_15_01} with $u=0$. The only difference in this case is that $i=j$ and $i=j-1$ both should be considered as the maximizer of $d(\mr q- \mr c_i)$. 
The rest of the argument is the same. 

\subsubsection{A remark for Section~\ref{sec:theorem3ab}} 
\label{sec:proof_Lemma_1_others}

A variation of Lemma~\ref{lm:corners_trivial} is used in Section~\ref{sec:theorem3ab} to prove Theorem~\ref{thm:limiting_process_step_shock} (a) (and (b) similarly). 
With the new $t=s_{j}N$ and an additional restriction $u_i>0$ for $1\le i\le k$, the change is that
the inequality~\eqref{eq:aux_2016_08_09_02} should be checked separately when $i=j-1$. In this case the left hand side of~\eqref{eq:aux_2016_08_09_02} is $0$. However, the first term on the right hand side of~\eqref{eq:aux_2016_08_09_02} is negative (since $u>0$) and at least of order
	\begin{equation}
	O(t^{2/3}N^{-1}\ttt_N^{-1})=O(t^{-1/3})
	\end{equation}
	which dominate the second term $O(t^{\frac13+\frac{2}{9}\epsilon} N^{-1})$. Therefore~\eqref{eq:aux_2016_08_09_02} still holds for sufficiently large $N$.

\subsection{Proof of Lemma~\ref{lm:corner_main2}}
\label{sec:proof_Lemma_3}

The first equation of Lemma~\ref{lm:corner_main2} is similar to Corollary 2.7 of \cite{Ferrari-Nejjar15}, which follows from a general theorem in the same paper (see Theorem 2.1 in \cite{Ferrari-Nejjar15}). 
We can apply this general theorem to our case. The only change from Corollary 2.7 of \cite{Ferrari-Nejjar15} is that in order to check Assumption 3 for Theorem 2.1 in \cite{Ferrari-Nejjar15}, we use Proposition~\ref{prop:tail_transversal}, which is a stronger tail estimate than used in \cite{Ferrari-Nejjar15}. 
This is needed since in our case $|\mr q|$ can be as large as $O(N^{3/2-\epsilon})$ which was only $O(N)$ in \cite{Ferrari-Nejjar15}.
Alternatively the first equation also follows by the same argument given below for the periodic TASEP.

For the second equation of Lemma~\ref{lm:corner_main2}, the result of Ferrari and Nejjar is not applicable directly 
due to the periodicity. This periodicity implies that the maximal paths $\pi_{\mr c_{j-1}}^{max}(\mr q)$ and $\pi_{\mr c_{j}}^{max}(\mr q)$ are not independent near the corners $\mr c_{j-1}$ and $\mr c_j$ respectively. Also note that these two paths may intersect near $\mr q$. To handle these dependencies, we need to consider the maximal paths with new starting (and ending) points such that the new paths are asymptotically independent, then compare the last passage times given by the original paths and the new ones.
This idea is in \cite{Ferrari-Nejjar15} in which the dependence near $\mr q$ was handled. 



Note that by using Propositions~\ref{prop:tail_transversal} and~\ref{prop:tail_transversal2}, both maximal paths from $\mr c_{j-1}$ and $\mr c_j$ to $\mr q$ are bounded in a strip with width of order $O(t^{2/3+\kappa\epsilon})\le  O(N^{1-\kappa\epsilon^2})$ with high probability, where $\kappa=4/9$ and $\epsilon$ is the constant defined in Theorem~\ref{thm:limiting_process_step} such that $t< C N^{3/2-\epsilon}$. Denote these two strips by $\Omega_{j-1}$ and $\Omega_j$.

We then pick lattice points $\mr p$ and $\mr p'$ neighboring to $\overline{\mr q\mr c_{j-1}}$, and $\mr r$ and finally $\mr r'$ neighboring to $\overline{\mr q\mr c_j}$, such that the following conditions are satisfied:

(1) The part of the strip $\Omega_{j-1}$ between $\mr p$ and $\mr p'$ does not ``intersect'' that of $\Omega_j$ between $\mr r$ and $\mr r'$, here we say a set $A\subseteq \realR^2$ does not ``intersect'' another set $B\subseteq\realR^2$ if and only if $A\cap (B+i\mr v)=\emptyset$ for all integer $i$. 

(2) $d(\mr q-\mr p)$, $d(\mr p'-\mr c_{j-1})$, $d(\mr q-\mr r)$, and $d(\mr r'-\mr c_j)$ are all bounded by $ctN^{-\kappa\epsilon^2}$.

(3) $d(\mr q-\mr p)+d(\mr p'-\mr c_{j-1})=d(\mr q-\mr r)+d(\mr r'-\mr c_j)+O(1)$.

\begin{figure}
	\centering
	\includegraphics[scale=0.35]{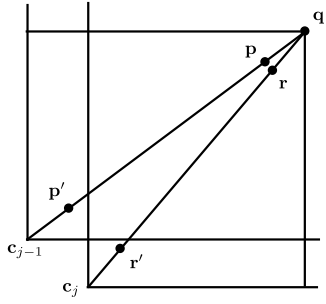}
	\caption{Illustration of the lattice points $\mr p,\mr p',\mr r,$ and $\mr r'$}
	\label{fig:two_corners}
\end{figure}

See Figure~\ref{fig:two_corners} for an illustration of these points.
To proceed, we need the following lemma. 
\begin{lm}[Lemma 4.2 in \cite{Ben_Arous-Corwin11}]
	\label{lm:Ben_Arous-Corwin}
	Assume $X_N\ge \tilde X_N$ and $X_N\Rightarrow D_1$ as well as $\tilde X_N\Rightarrow D_1$; and similarly $Y_N \ge \tilde Y_N$ and $Y_N\Rightarrow D_2$ as well as $\tilde Y_N\Rightarrow D_2$. 
	Then if $\max\{\tilde X_N,\tilde Y_N\}\Rightarrow D_3$, we also have $\max\{X_N,Y_N\}\Rightarrow D_3$.
\end{lm}

Set
\begin{equation}
\begin{split}
X_N&=\frac{H_{\mr c_{j-1}}(\mr q)-d(\mr q-\mr c_{j-1})-H_{\mr c_{j-1}}(\mr p')+d(\mr p'-\mr c_{j-1})-H_{\mr p}(\mr q)+d(\mr q-\mr p)}{\arao^{-1}s(\mr q-\mr c_{j-1})},\\
\tilde X_N&= \frac{H_{\mr p'}(\mr p)-d(\mr p-\mr p')}{\arao^{-1}s(\mr q-\mr c_{j-1})},\\
Y_N&=\frac{H_{\mr c_{j}}(\mr q)-d(\mr q-\mr c_{j})-H_{\mr c_{j}}(\mr r')+d(\mr r'-\mr c_{j})-H_{\mr r}(\mr q)+d(\mr q-\mr r)}{s(\mr q-\mr c_j)},\\
\tilde Y_N&=\frac{H_{\mr r'}(\mr r)-d(\mr r-\mr r')}{s(\mr q-\mr c_j)},
\end{split}
\end{equation}
where $\arao=\arao(N)$ is defined in~\eqref{eq:aux_2016_08_24_03}.
Then from~\eqref{eq:aux_2016_08_21_02} and the definition of $\mr p,\mr p',\mr r, \mr r'$,  these random variables satisfy the conditions of Lemma~\ref{lm:Ben_Arous-Corwin} with 
$D_1$ being a GUE Tracy-Widom random variable times $r=\lim_{N\to\infty}R$ and $D_2$ also being a GUE Tracy-Widom distribution. 
Now we note that Proposition~\ref{prop:tail_transversal2} implies that the maximal path from $\mr p'$ to $\mr p$ and the maximal path  from $\mr r'$ to $\mr r$ stay in $\Omega_{j-1}$ and $\Omega_j$ respectively with high probability.
Therefore the two random variables $H_{\mr p'}(\mr p)$ and $H_{\mr r'}(\mr r)$ are independent with high probability. 
Hence 
we have 
\begin{equation}
\lim_{N\to\infty}\prob_{\mr v}\left(\max\{\tilde X_N,\tilde Y_N\} \le x\right)=\FGUE(x)\FGUE(r^{-1}x),
\end{equation}
where $r$ comes from the ratio between $s(\mr c_j)$ and $s(\mr c_{j-1})$.
Now from Lemma~\ref{lm:Ben_Arous-Corwin}, we obtain 
\begin{equation}
\label{eq:aux_2016_08_24_02}
\lim_{N\to\infty}\prob_{\mr v}\left(\max\{X_N,Y_N\} \le x\right)=\FGUE(x)\FGUE(r^{-1}x).
\end{equation}

The second equation of Lemma~\ref{lm:corner_main2} we would like to prove can be written as 
\begin{equation}
\label{eq:aux_2016_08_24_01}
	\lim_{N\to\infty}\prob_{\mr v}\left(\max\left\{X_N^*,Y_N^*	\right\}\le x\right)=\FGUE(x)\FGUE(r^{-1}x).
\end{equation}
with
\begin{equation}
X_N^*=\frac{H_{\mr c_{j-1}}(\mr q)-d(\mr q-\mr c_{j-1})}{\arao^{-1}s(\mr q-\mr c_{j-1})},\qquad Y_N^*=\frac{H_{\mr c_{j}}(\mr q)-d(\mr q-\mr c_{j})}{s(\mr q-\mr c_j)}.
\end{equation}
In order to derive~\eqref{eq:aux_2016_08_24_01} from~\eqref{eq:aux_2016_08_24_02}, we note that 
Lemma~\ref{prop:tail_estimate} and Proposition~\ref{prop:tail_transversal} imply that 
\begin{equation}
\lim_{N\to\infty} \prob_{\mr v}\left(|X_N-X_N^*|>\epsilon\right)=0, \qquad\lim_{N\to\infty} \prob_{\mr v}\left(|Y_N-Y_N^*|>\epsilon\right)=0
\end{equation}
for arbitrary $\epsilon>0$. 
Using the simple inequality,
\begin{equation}
\left|\max\{X_N^*,Y_N^*\}-\max\{X_N,Y_N\}\right|\le \max\{|X_N-X_N^*|, |Y_N-Y_N^*|\},
\end{equation}
we obtain that the left hand sides of~\eqref{eq:aux_2016_08_24_02} and~\eqref{eq:aux_2016_08_24_01} are equal. Thus~\eqref{eq:aux_2016_08_24_01} follows.

\subsection{Proof of Lemma~\ref{lm:flat_trivial}}
\label{sec:proof_Lemma_flat_1}
We note that 
\begin{equation}
d(\mr q-\tilde{\mr c}_i) = \left(\sqrt{(1-\rho)\left(t-\rho^{-1}(N-i)\right)}+\sqrt{N-i}\right)^2+o(N).
\end{equation}
As a function of $i$, its maximum occurs at $i= N-\rho^2 t$. 
Hence if $j$ is not in $I$, i.e., $|j-N+\rho^2 t|\ge N/4$, then $d(\mr q-\tilde{\mr c}_j)$ is less than the maximum of $d(\mr q-\tilde{\mr c}_i)$, $i\in I$, and the difference is of at least $O(N)$. 
More rigorously, we write
\begin{equation}
\rho\cdot d(\mr q-\tilde{\mr c}_i) = \rho t-\left(\sqrt{\rho^2 t-\rho(N-i)}-\sqrt{(1-\rho)(N-i)}\right)^2+o(N).
\end{equation}
If $N-i=\rho^2 t+O(1)$, the above equation equals to $\rho t+o(N)$. On the other hand, if $N-i\le \rho^2t-N/4$, we have
\begin{equation}
\sqrt{\rho^2 t-\rho(N-i)}-\sqrt{(1-\rho)(N-i)}\ge \sqrt{\rho^2(1-\rho)t+\frac{\rho N}{4}}-\sqrt{\rho^2(1-\rho)t-\frac{(1-\rho)N}{4}}
\end{equation}
which is at least $O(N(t+N)^{-1/2})$. Therefore we obtain
\begin{equation}
\label{eq:aux_2016_08_15_05}
\rho\cdot d(\mr q-\tilde{\mr c}_i)\le \rho^2 t- cN^2(t+N)^{-1}+o(N)
\end{equation}
for some $c>0$. Similarly if $\rho^2t+N/4\le N-i\le \rho t$, we have the same bound~\eqref{eq:aux_2016_08_15_05}. Note that $N^2(t+N)^{-1}\ge  t^{1/3+\epsilon}$. By using~\eqref{eq:aux_2016_08_15_05} and Lemma~\ref{lm:comparison_DLPP}, we obtain Lemma~\ref{lm:flat_trivial}.

\appendix
\section{Density profile of TASEP with periodic step initial condition}
\label{sec:appendix}

In this appendix, we summarize the macroscopic picture of the periodic TASEP and the infinite TASEP with periodic step initial condition~\eqref{eq:step_ic} via solving the Burger's equation. 
We state the density profile, and the locations of the shock and any given particle as time $t$ without much details since the computation is standard. 
Furthermore we do not study the issue of the convergence in the hydrodynamic limit to the Burger's solution; the computations in this Appendix are used only to provide intuitive ideas and are not used in the proofs of the theorems.


We assume that $0<\rho\le \frac12$. 
Consider the Burger's equation for the infinite TASEP
\begin{equation}
\frac{\dd }{\dd t}q(x;t) +\frac{\dd }{\dd x}\left(q(x,t)(1-q(x,t))\right)=0
\end{equation}
with the periodic initial condition
\begin{equation}
q(x;0)=\begin{dcases}
1,\qquad &-\rho\le x-[x]-1\le 0,\\
0,\qquad &0< x-[x]< 1-\rho,
\end{dcases}
\end{equation}
where $[x]$ means the largest integer which is less than or equal to $x$. 
The entropy solution $q(x;t)$ represents the local density profile at location $xL$ and time $tL$. Note that the solution is also periodic, $q(x+1, t)=q(x,t)$, and hence $q(x;t)$ also represents the local density profile for the periodic TASEP with the same initial condition.

We now solve the above Burger's equation explicitly. 
Due to the periodicity, we state the formula of $q(x;t)$ for $x$ only in an interval of length $1$.

For time $t\le \frac1{4\rho}$, there is no shock and the solution is given by the following: 
For $0\le t\le \rho$,
\begin{equation}
q(x;t)=\begin{dcases}
1,& -\rho\le x\le -t,\\
\frac{1}{2}-\frac{1}{2t}x,& -t<x<t,\\
0,& t\le x<1-\rho.
\end{dcases}
\end{equation}
For $\rho \le t\le \frac{1}{4\rho}$,
\begin{equation}
q(x,t)=	\begin{dcases}
\displaystyle
\frac{1}{2}-\frac{1}{2t}x,	& -2\sqrt{\rho t} +t\le x\le t,\\
0,	&t<x< -2\sqrt{\rho t} +t+1.
\end{dcases}
\end{equation}

The shocks are generated at time $t=\frac{1}{4\rho}$ at the locations $\frac1{4\rho}+\intZ$. 
(In terms of the TASEP, the above time corresponds to the time $\frac{1}{4\rho}L=\frac{1}{4\rho^2}N$.)
Let us denote by $x_{\rm s}(t)$ the location of the shock of the Burger's equation at time $t$ 
which was initially generated at the location $-1+\frac1{4\rho}$, i.e.  
$x_{\rm s}(\frac1{4\rho})=-1+\frac1{4\rho}$. 
One can find that the shock location is given by 
\begin{equation} \label{eq:shockloca}
	x_{\rm s}(t)=-\frac{1}{2}+(1-2\rho)t 
\end{equation}
and the density profile is given by 
\begin{equation}
q(x;t)=\frac12-\frac{1}{2t}x, \quad x_{\rm s}(t)\le x < x_{\rm s}(t)+1
\end{equation}
for all $t\ge \frac{1}{4\rho}$.
This shows that the density profile difference at the shock, $\Delta q_{\rm s}(t):= \lim_{x\to x_{\rm s}(t)^+} q(x; t)-\lim_{x\to x_{\rm s}(t)^-} q(x; t)$, is given by $\Delta q_{\rm s}(t)=\frac1{2t}$ at time $t\ge \frac{1}{4\rho}$. 
As $t\to\infty$, this gap tends to zero and $q(x;t)\to \rho$ for all $x\in \realR$.
However, the density profile is not yet ``flat enough'' when $t\ll L^{1/2}$ (which corresponds to the sub-relaxation time scale $t\ll L^{3/2}$ in TASEP). 
Indeed, note that that when $t\ll L^{1/2}$, the gap satisfies $ \Delta q_{\rm s}(t)\gg \frac1{L^{1/2}}$ (and 
 the absolute value of the slope of the density profile at continuous points is $\gg \frac1{L^{1/2}}$.)
In terms of the TASEP scale of time and space, $\Delta q_{\rm s}(t) L \gg L^{1/2} \gg (tL)^{1/3}$ which means that the gap is greater than the KPZ height fluctuations. 

Given the formula of the density profile, we can compute the 
expected location of the $[\alpha N]$-th particle (the one initially located at $-N+[\alpha N]$) heuristically. Here $\alpha$ is an arbitrary constant satisfying $0<\alpha\le 1$. This particle meets a shock at the discrete (rescaled by $L$) times
\begin{equation}
\frac{\left(\sqrt{k-\alpha} +\sqrt{k+1-\alpha}\right)^2}{ 4\rho }, \qquad k=1,2,\cdots.
\end{equation}
The particles location (rescaled by $L$) is heuristically given by 
\begin{equation}
	\label{eq:aux_002}
	\begin{split}
	X_{\alpha}(t)	&=(\sqrt{t}-\sqrt{(k+1-\alpha)\rho})^2-(k+1-\alpha)\rho+k,\\
	&=t(1-\rho)+(\sqrt{t\rho}-\sqrt{k+1-\alpha})^2-(1-\rho)(1-\alpha)+X_{\alpha}(0)
	\end{split}
\end{equation}
for time satisfying
\begin{equation}
	\label{eq:aux_003}
	\frac{\left(\sqrt{k-\alpha} +\sqrt{k+1-\alpha}\right)^2}{ 4\rho }\le t <\frac{\left(\sqrt{k+1-\alpha} + \sqrt{k+2-\alpha}\right)^2}{ 4\rho }.
\end{equation}


\def\cydot{\leavevmode\raise.4ex\hbox{.}}

\end{document}